\documentclass[12pt]{article}
\usepackage{amsfonts}
\usepackage{mathrsfs}
\usepackage{epsfig}
\usepackage{amssymb,amsmath,amsthm,amscd}
\usepackage{latexsym}
\usepackage{xcolor,bm}
\usepackage{xcolor,bm}
\usepackage{cite}
\usepackage{color} 
 \theoremstyle{theorem}
 \newtheorem{thm}{Theorem}[section]
 \newtheorem{lemma}{Lemma}[section]
  
  \newtheorem{remark}{Remark}[section]
 \newtheorem{defn}{Definition}[section]
 \newtheorem{ex}{Example}[section]
 
 \newtheorem{cor}{Corollary}[section]

\newcommand{\cA}{{\mathcal A}}
\newcommand{\cB}{{\mathcal B}}
\newcommand{\cC}{{\mathcal C}}

\newcommand{\cE}{{\mathcal E}}

\newcommand{\cH}{{\mathcal H}}
\newcommand{\cK}{{\mathcal K}}

\newcommand{\cM}{{\mathcal M}}
\newcommand{\cN}{{\mathcal N}}

\newcommand{\cS}{{\mathcal S}}

\newcommand{\cU}{{\mathcal U}}
\newcommand{\cV}{{\mathcal V}}

\newcommand{\mB}{{\mbox B}}

\newcommand{\sA}{{\mathscr A}}

\newcommand{\sC}{{\mathscr C}}

\newcommand{\sM}{{\mathscr M}}

\newcommand{\sX}{{\mathscr X}}
\newcommand{\sS}{{\mathscr S}}

\def\R{\mathbb{R}}

\def\1{\mathbb{1}}

\def\bc{\begin{center}}
\def\ec{\end{center}}
\def\be{\begin{equation}}
\def\ee{\end{equation}}
\def\ba{\begin{array}}
\def\ea{\end{array}}
\def\benu{\begin{enumerate}}
\def\eenu{\end{enumerate}}
\def\bt{\begin{thm}}
\def\et{\end{thm}}
\def\bl{\begin{lemma}}
\def\el{\end{lemma}}
\def\bco{\begin{cor}}
\def\eco{\end{cor}}
\def\br{\begin{remark}}
\def\er{\end{remark}}
\def\bd{\begin{defn}}
\def\ed{\end{defn}}
\def\bp{\begin{proposition}}
\def\ep{\end{proposition}}
\def\bo{\begin{proof}}
\def\eo{\end{proof}}
\def\bx{\begin{example}}
\def\ex{\end{example}}

\def\pa{\partial}
\def\a{\alpha}
 \def\de{\delta}

\def\lam{\lambda} \def\Lam{\Lambda}

\def\ve{\varepsilon}
\def\sig{\sigma}\def\Sig{\Sigma}
\def\vp{\varphi}
\def\w{\omega}
\def\gam{\gamma}\def\Gam{\Gamma}\def\e{{\rm e}}

\def\~{\widetilde}
\def\A{\forall}
\def\ol{\overline}
\def\ul{\underline}
\def\Cap{\bigcap}
\def\Cup{\bigcup}
\def\ra{\rightarrow}

\def\stac{\stackrel}
\def\8{\infty}
\def\X{\times}

\def\mb{\mbox}

\def\ss{\subset}
\def\emp{\emptyset}

\def\.{\cdot}
\def\leq{\leqslant}\def\geq{\geqslant}\def\d{{\rm d}}\def\no{\nonumber}

\def\Hs{\hspace{0.8cm}}
\def\hs{\hspace{0.4cm}}
\def\Vs{\vskip10pt}
\def\vs{\vskip5pt}

\def\[{\left[}
\def\]{\right]}
\def\({\left(}
\def\){\right)}

\title{Dynamic Bifurcation of Nonautonomous Evolution Equations: Invariant Manifold Methods
}

\author{
\small{Chunqiu Li}, \footnote{Corresponding author. E-mail: lichunqiu@wzu.edu.cn}
\Hs{Jintao Wang} \footnote{E-mail: wangjt@wzu.edu.cn}  ~
\\\\
\small\it Department of Mathematics, Wenzhou University,\\ \small\it Wenzhou, Zhejiang Province, 325035, P. R. China\\
}
\date{\small\today}

\begin{document}

\maketitle
\begin{abstract}\baselineskip 14pt
In this paper, we use the global invariant manifold and the reduced singular cohomology groups method, which is different from those in the literature, to study the dynamic bifurcation from infinity of the nonautonomous evolution equation in terms of invariant sets by applying the Conley index theory (due to Rybakowski). We first establish a nonautonomous global invariant manifold for the abstract equation, which allows us to reduce the original system to this finite-dimensional manifold. Then, a homotopy between the reduced equation and a product flow is constructed. Finally, by considering the reduced singular cohomology theory of the Conley index, we establish our main theorems on dynamic bifurcations from infinity for this nonautonomous equation. As an example, a nonautonomous parabolic equation on unbounded domains is considered. Some new detailed results on bifurcations from infinity of the parabolic equation under an appropriate Landesman-Lazer type condition are proved, including the existence of a nonautonomous Morse decomposition, and improving the earlier works in the literature.
\end{abstract}


\setcounter {equation}{0}
\section{Introduction}

\hskip 6mm It is well-known that the theory of dynamic bifurcation plays an important role in nonlinear
problems, and has aroused much interest in the past decades; see e.g., \cite{APR,Chow,CLR,LWZH,MW,Ras}, etc. There are many significant works on dynamic bifurcations of evolution equations; see e.g., \cite{ACMV,LLZ,LLW,LCW,MW,San,W1,ZL} for autonomous systems and \cite{CKMNO,DNO,DNO1,KR,M} for some scalar nonautonomous equations, and some canonical PDEs (see e.g., \cite{CLR,CLR12,LCLM,LW25}, etc.). However, compared with the autonomous case, the results on nonautonomous evolution equations especially for nonautonomous PDEs are few, as the dynamics of nonautonomous systems is much more complicated. Therefore, it is of great importance to develop some appropriate theories  to investigate the bifurcation of nonautonomous equations.

In this article, we are concerned with the following nonautonomous evolution equation:
\be\label{e1.1}
\frac{\d u}{\d t}+A u=\lam u+f(u)+g(t)
\ee
on a Banach space $X$, where $A$ is a sectorial operator on $X$, $\lam\in\R$ is the bifurcation parameter, and $f$ is a global Lipschitz continuous mapping from $X^\a\X\R$ ($0\leq \a<1$) to $X$ and is sublinear as $\|u\|_\a\ra \8$ uniformly on bounded $\lam$-intervals. We are mainly interested in the dynamic bifurcation from infinity of this nonautonomous equation.

The topic on bifurcations from infinity of evolution equations can be traced back to the earlier work of Rabinowitz \cite{R}, where the author studied the problem of solutions of operator equations. Later, the result was partially extended 
to potential operator equations (see e.g., \cite{SW,T1,DH}, etc.) and elliptic equations (see e.g., \cite{KS,SZ}, etc.). Recently,
\'Cwiszewski and Kryszewski \cite{CK} investigated the bifurcation from infinity of elliptic equations on unbounded domains $\R^N$. By utilizing the Conley index theory and studying the corresponding dynamic equation on unbounded domains:
$$
u_t-\Delta u+V(x)u=\lam u+f(x,u),\hs x\in \mathbb{R}^N,
$$
they showed that if the bifurcation parameter is an eigenvalue of the Hamiltonian, then the bifurcation from infinity of this equation occurs under some Landesman-Lazer type conditions. Later, Li and Wang \cite{LW21}, using the invariant manifold and the Conley index theory, established some results on bifurcations from infinity and multiplicity of solutions of the Schr\"odinger equation. The interested reader is referred to \cite{KS,LW25-2,G,Stu,SZ}, etc. and the references therein, for examples on the bifurcation from infinity for elliptic equations on unbounded domains.

For a nonautonomous system, invariant sets (or bounded full solutions) play an important role in understanding its dynamical behaviors. This is because that the dynamics of a system is usually captured by all the invariant sets. Thus it is essential to study the bifurcation of nonautonomous systems in terms of invariant sets. A very efficient tool to study invariant sets or bounded full solutions is the theory of the Conley index \cite{Con}, which is extended by Rybakowski \cite{Ry} to semiflows in infinite-dimensional spaces. There are many important results on bifurcations of invariant sets established by using the Conley index theory; see e.g., \cite{BS18,FX,LLZ,LW18,Ry,San,W1,W2,ZL}, etc. for autonomous differential equations. However, for the nonautonomous PDE, there are very few woks in this line because of its high complexity. Very recently, Li and Wang \cite{LW25} investigated the dynamic bifurcation of nonautonomous evolution equations on bounded domains, and presented some results on bifurcations from infinity by using the reduced singular cohomology theory of the Conley index. Based on this work, in this article we further use the nonautonomous invariant manifold theory to study the bifurcation from infinity in terms of invariant sets for \eqref{e1.1}. To the best of our knowledge, there are no results on studying the bifurcation from infinity of nonautonomous PDEs on unbounded domains by employing the nonautonomous invariant manifold.

Compared with the evolution equation on bounded domains discussed in \cite{W2,LW18,LW25}, the operator $A$ in \eqref{e1.1} may not have a compact resolvent. For instance, if $A=-\Delta+v(x)$, $x\in \R^N$, as discussed in \cite{CK,KS}, then the bifurcation problems in this case (on unbounded domains) are more difficult. The underlying reason is that the spectrum of the operator $A$ is not discrete in general, and may be quite complicated. Moreover, the verification of the compactness for the skew-product flow generated by \eqref{e1.1} is usually not an easy task (see e.g. \cite{Pri}).
To overcome these difficulties, we first establish a nonautonomous global invariant manifold of \eqref{e1.1}, which allows us to reduce \eqref{e1.1} to this finite-dimensional manifold. Then, we apply the Conley index theory (see \cite{Ry}) to study the bifurcation from infinity of this nonautonomous finite-dimensional reduced equation.

As we all know, the Conley index theory can not be directly applied to the nonautonomous system, as the solution operators can not generate a semiflow. Fortunately, one can define a skew-product flow associated with the nonautonomous equation, which is a semiflow in the corresponding product space; see e.g., \cite{J,J1,LSS,Pri,Ward,W92}, etc. However, in the framework of skew-product flows in the product space, it is very difficult to compute the Conley index of invariant sets. To the end, we develop the method in \cite{Ward,W92} to construct a homotopy between the skew-product flow and product flows for the reduced system, and then use some developed topological consequences on the reduced singular cohomology groups (see the Appendix) to discuss the Conley index of invariant sets.
Based on this homotopy and the Conley index of maximal compact invariant sets, we establish our main results of dynamic bifurcations from infinity on invariant sets for \eqref{e1.1} .

As an example, we study the nonautonomous parabolic equation on unbounded domain:

\be\label{e1.3}
u_t-\Delta u+v(x)u=\lam u+f(x,u)+g(x,t),\hs x\in \mathbb{R}^N,
\ee
where $v\in L^\infty(\R^N)$, $N\geq 1$, $\lam\in \R$ is the bifurcation parameter, and $f$ satisfies some Lipschitz condition 
and the following Landesman-Lazer type condition (see also \cite{LSS}):
\be\label{LL}
  \liminf\limits_{s\rightarrow+\infty}f(x,s)\geqslant\ol{f}>0,\Hs
  \limsup\limits_{s\rightarrow-\infty}f(x,s)\leqslant-\underline{f}<0
  \ee uniformly with respect to $x\in \R^N,$ where $\ol f$ and $\ul f$ are independent of $x$.

Firstly, we give a detailed descriptions on the dynamical behaviors of the nonautonomous reduced equation, and present a more deeper level of  discussion on the bifurcation from infinity of this equation near an isolated eigenvalue $\mu$ of the operator $A=-\Delta+v(x)$. Specifically, it will be shown that there exists $\eta>0$ such that if $\lam\in \Lam_1=[\mu-\eta,\mu)$, then the maximal compact invariant set $K_\lam$ of the skew-product flow $\Phi_\lam$ generated by this equation has a Morse decomposition $\sM=\{K_\lam^1,K_\lam^\8\}$ with $K_\lam^1$ being bounded on $\Lam_1$, while
\be\label{e4.29e}
\lim_{\lam\ra\mu}\min_{\~w\in K_\lam^\8}\|w\|=\8,
\ee
where $\~w=(w(t,0;u_0,p),\theta_tp)$ and $K_\lam^\8$ is an attractor for the skew-product flow $\Phi_\lam$. Moreover, we prove that
both of the sets $\sS^1$ and $\sS^\8$ have a component $\Gam$ satisfying $\Gam[\lam]\neq \emp$ for each $\lam\in\Lam_1$, where
$$
  \sS^1=\ol{\Cup_{\lam\in\Lam_1}(K_\lam^1\X\{\lam\})},\hs \sS^\8=\ol{\Cup_{\lam\in\Lam_1}(K_\lam^\8\X\{\lam\})}.
  $$
More interestingly, we further show that for each $\lam\in\Lam_1$, $K_\lam$ also has a Morse decomposition $M=\{\cA_\lam^\8,\cA_\lam^1\}$ in the sense of the works in Aragao-Costa, Caraballo, Carvalho and Langa \cite{ACCL} and Rasmussen \cite{Ras,Ras07}, with
$$
  \cA_\lam^\8=\Cup_{p\in\cH}\big(A_\lam^\8(p)\X\{p\}\big),\hs \cA_\lam^1=\Cup_{p\in\cH}\big(A_\lam^1(p)\X\{p\}\big)
  $$
and the family of sets $\{A_\lam^\8(p)\}_{p\in \cH}$ being pullback attractors for the reduced system.
Finally, the corresponding results on dynamic bifurcations from infinity of the original equation \eqref{e1.3} are derived.

It is worth mentioning that our approach, using the invariant manifold and the reduced singular cohomology groups to study the dynamic bifurcation from infinity of \eqref{e1.1} is different from those in the literature. Precisely, we first develop the idea from \cite{LW21} (or \cite{LCW}) on autonomous PDEs to establish a nonautonomous global invariant manifold for \eqref{e1.1}. Then we extend some techniques from \cite{Ward,W92} and construct a homotopy for the obtained finite-dimensional nonautonomous reduced equation. Finally, we use some topological consequences on the reduced singular cohomology groups to establish our main results for \eqref{e1.1}, improving and extending some results established in \cite{CK} (see Remark \ref{r6.1}), \cite{LW25,W2,Ward,W92} and other earlier works in the literature. Besides, our invariant manifold method can also be applied to the case that the operator $A$ in \eqref{e1.1} has a compact resolvent (see Section 7 for details). Finally,
if, instead of \eqref{LL}, we assume that
\be\label{1.5}
  \limsup\limits_{s\rightarrow+\infty}f(x,s)\leq-\ol{f}<0,\Hs
  \liminf\limits_{s\rightarrow-\infty}f(x,s)\geq\underline{f}>0
  \ee
uniformly with respect to $x\in\R^N$, then the ``dual" versions of all our results also hold true.

This work is organized as follows. Section 2 is concerned with some
preliminaries. In Section 3 we establish a nonautonomous global invariant manifold of the nonautonomous equation \eqref{e1.1} on $\R^N$, which allows us to reduce \eqref{e1.1} to this finite-dimensional manifold. In Section 4, we first construct a homotopy between the skew-product flow generated by the nonautonomous reduced equation and a product flow, and then establish our main results on the bifurcation from infinity of the reduced system. In Section 5, the dynamic bifurcation from infinity of the nonautonomous parabolic equation \eqref{e1.3} is studied. We give a detailed descriptions on the dynamical behaviors of the corresponding reduced equation, and establish some new results on the bifurcation from infinity of the reduced equation. Section 6 is devoted to presenting the corresponding results on dynamic bifurcations from infinity for the original parabolic equation. In Section 7, we summarize this work and provide a remark. Finally, in the Appendix part, we give some topological consequences on the reduced singular cohomology groups.

\setcounter {equation}{0}
\section{Preliminaries}

\hs \, In this section we first make some preliminaries.
\subsection{Semiflows}
\hs \, For the reader's convenience, we collect some fundamental notions on dynamical systems on metric spaces; see \cite{CLR,Ry} for details.
Let $X$ be a complete metric space.

A nonautonomous dynamical system consists of a ``base flow" on a metric space and a ``cocycle semiflow" on a phase space which is in some sense driven
by the base flow.
\vs
A {\it base flow} $\theta=\{\theta_t\}_{t\in \R}$ is a dynamical system on a metric space $\cH$, i.e., a group of homeomorphisms from $\cH$ to itself so that
\benu
\item[(i)] $\theta_0={\rm id}_{\cH}$;
\item[(ii)] $\theta_t\theta_s=\theta_{t+s}$, for all $t,s\in \R$;
\item[(iii)] the mapping $(t,p)\ra \theta_tp$ is continuous.
\eenu
The metric space $\cH$ is usually called a {\em base space}.
\vs
A {\it cocycle semiflow} $\psi$ on the space $X$ over $\theta$ is a continuous mappings
$\psi:\R_+\X \cH\X X\ra X$ such that
\benu
\item[(i)] $\psi(0,p)={\rm id}_X$ for all $p\in \cH$;
\item[(ii)] $\psi(t+s,p)=\psi(t,\theta_sp)\psi(s,p)$ for all $t,s\in\R_+$ and $p\in \cH$.
\eenu

Given a nonautonomous dynamical system $(\psi,\theta)$ on the space $X\X \cH$. Define a mapping $\Psi$ by
$$
  \Psi(t)(u,p)=(\psi(t,p)u,\theta_tp),\Hs \A(u,p)\in X\X \cH,\hs t\geqslant 0.
  $$
Clearly, $\Psi$ is a global semiflow on $X\X \cH$, which is usually called a {\it skew-product flow}, and the space $X\X \cH$ is called a {\it product space}.
For the sake of convenience, we set $\sX=X\X \cH$ and rewrite $\Psi(t,\~u)$ as $\Psi(t)\~u$ for $\~u\in\sX$.
\vs
We call a continuous mapping $\gam:\R\ra \sX$ a {\it full trajectory} (or {\em full solution}) of $\Psi$ on $\R$, if
$$
  \gam(t)=\Psi(t-s)\gam(s),\hs \forall t,s\in \R,\,\,\, t\geq s.
  $$
The orbit of a full trajectory $\gam$ is defined to be the following set $${\rm orb}(\gam)=\{\gam(t):t\in\R\},$$ which is called {\it a full orbit}.

Let $\cU\subset \sX$.  By $S_\8(\Psi,\cU)$ we denote the union of all bounded full orbits of $\Psi$ in $\cU$. If $\cU=\sX$, set
$$S_\8(\Psi,\sX)=S_\8(\Psi).$$

The {\em $\omega$-limit set} $\w(\gam)$ and {\em $\w^*$-limit set} $\w^*(\gam)$ of a full trajectory $\gam$ are defined, respectively, by
$$\ba{ll}
\w(\gamma)=\{z\in \sX:\,\,\,\,\mb{there is }  t_n\ra \8 \mb{ such that }\gamma(t_n)\ra z\},\ea$$
$$
\ba{ll}
\w^*(\gamma)=\{z\in \sX:\,\,\,\,\mb{there is }  t_n\ra -\8 \mb{ such that }\gamma(t_n)\ra z\}.\ea
$$
A set $\cS\ss \sX$ is said to be {\em invariant}, if $\cS=S_\8(\Psi,\cS)$. A compact invariant set $\cA\ss \sX$ is called an {\em attractor} of  $\Psi$, if it attracts a neighborhood $\cU$ of itself, namely,
$$
\lim_{t\ra\8}\d_H\(\Psi(t)\cU,\cA\)=0.
$$
A compact invariant set $\cS$ of $\Psi$ is said to be {\em isolated}, if
there exists a neighborhood $\cN$ of $\cS$ such that $\cS=S_\8(\Psi,\ol{\cN})$.
Correspondingly, the set $\cN$ is called an {\em isolating neighborhood} of $\cS$.


\begin{defn}\label{defn2.3}(\!\!\cite{Ry})
A set $\cN\subset \sX$ is said to be admissible (w.r.t. $\Psi$), if for every sequences $\~u_n\in
\cN$ and $t_n\rightarrow\infty$ satisfying $\Psi([0,t_n])\~u_n\subset \cN$ for
all $n$, then the sequence $\Psi(t_n)\~u_n$ has a convergent subsequence.

\end{defn}

\subsection{Conley index due to Rybakowski}
\hs For the readers' convenience, we finally recall the definition of Conley index; see
\cite{Con} and \cite{Ry}, etc., for details.

Let $\Psi$ be a local (or global) semiflow on $\sX$. 



Let $\cB$ be a bounded closed subset of $\sX$. $\~u\in \pa \cB$ is called a {\em strict ingress} (resp., {\em strict egress}, {\em bounce-off}) point of $\cB$, if for each trajectory $\gamma:[-\tau,s]\ra \sX$ with $\gamma(0)=\~u$, where $\tau\geq0$, $s>0$, the following conditions are satisfied:
\vs
\noindent(1) there is  $0<\ve<s$ such that
$$
\gamma(t)\in \mb{int}\cB\,\,\,(\mb{resp., }\, \gamma(t)\not\in \cB,\,\,\,\mb{resp., }\,\gamma(t)\not\in \cB),\Hs \A\,t\in (0,\ve);
$$
(2) if $\tau>0$, then there exists $\de \in (0,\tau)$ such that
$$
\gamma(t)\not\in \cB \,\,\,(\mb{resp., }\, \gamma(t)\in \mb{int}\cB,\,\,\,\mb{resp., }\,\gamma(t)\not\in \cB),\Hs \A\,t\in (-\de, 0).
$$
By $\cB^i$ (resp., $\cB^e$, $\cB^b$), we denote the set of all strict ingress (resp., strict egress, bounce-off) points of the closed set $\cB$, and write $\cB^-=\cB^e\cup \cB^b$.

A set $\cB\subset \sX$ is called an {\em isolating block} \cite{Ry} if $\cB^-$ is closed and $
\pa \cB=\cB^i\cup \cB^-.$
\vs
Suppose that $\cN,\cE$ are two closed subsets of $\sX$.  $\cE$ is called an {\em exit
set} of $\cN$, if the following properties hold:
\benu
\item[(1)] $\cE$ is $\cN$-{\em positively invariant}, that is, for each $\~u\in \cE$ and $t\geq 0$,
$$\Psi([0,t])\~u\subset \cN \Longrightarrow\Psi([0,t])\~u\subset \cE;$$
\item[(2)] for each $\~u\in \cN,$ if $\Psi(t_1)\~u\not\in \cN$ for some
$t_1>0$, then there exists $t_0\in[0,t_1]$ so that $\Phi(t_0)\~u\in \cE$.
\eenu

Let $\cS$ be a compact isolated invariant set of $\Psi$. A pair of bounded closed
subsets $(\cN,\cE)$ is called an {\em index pair} of $\cS$, if (1)\, $\cN\setminus \cE$ is an isolating neighborhood of
$\cS;$ (2)\, $\cE$ is an exit set of $\cN$.

We infer from \cite{Ry} that if $\cB$ is a bounded isolating block, then $(\cB,\cB^-)$ is an index pair of the maximal compact invariant set $\cS$  in $\cB$.

\begin{defn}\label{defn2.8}
Let $(\cN,\cE)$ be an index pair of $\cS$. The homotopy Conley index
of $\cS$ is defined to be the homotopy type $[(\cN/\cE,[\cE])]$ of the
pointed space $(\cN/\cE,[\cE])$, denoted by $h(\Psi,\cS).$
\end{defn}

Let $\{\Psi_{\lam}\}_{\lam\in \Lam}$ be a family of semiflows on $\sX$, where $\Lam$ is a metric space. We say that $\Psi_\lam$ {\it depends on $\lam$ continuously}, if $\Psi_\lam(t)\~u$ is defined at the point $(t,\~u,\lam)$, and for every sequence $(t_n,\~u_n,\lam_n)$ with $(t_n,\~u_n,\lam_n)\ra (t,\~u,\lam)$ as $n\ra\8$, then $\Psi_{\lam_n}(t_n)\~u_n$ is defined for all $n$ sufficiently large, and
$$\Psi_{\lam_n}(t_n)\~u_n\rightarrow \Psi_\lam(t)\~u\hs \mb{ as }\,n\rightarrow+\infty.$$

A set $\cN\ss \sX$ is {\it admissible w.r.t. $\Psi_\lam$}, if for every sequences $t_n\in [0,\8)$, $\~u_n\in \sX$ and $\lam_n\in\Lam$ with $t_n\ra \8$, $\lam_n\ra \lam_0$ (in $\Lam$) and $\Psi_{\lam_n}([0,t_n])\~u_n\subset \cN$, the sequence $\Psi_{\lam_n}(t_n)\~u_n$ has a convergent subsequence.
\vs
Let $\mathcal{I}(\sX)$ denote the family of pairs $(\Psi,\cK)$, where $\Psi$ is a semiflow on $\sX$ and $\cK$ is an isolated invariant set of $\Psi$, having an admissible (w.r.t. $\Psi$) isolated neighborhood.

In the following, we collect several important properties of Conley index from \cite{Ry}.
\benu
\item[(P1)] For each $(\Psi,K)\in \mathcal{I}(\sX)$, if $h(\Psi,K)\neq \ol{0}$, then $K\neq \emp$;
\item[(P2)] If $(\Psi,K_1),(\Psi,K_2)\in \mathcal{I}(\sX)$ and $K_1\cap K_2=\emp$, then $(\Psi,K_1\cup K_2)\in \mathcal{I}(\sX)$ and
$$h(\Psi,K_1\cup K_2)=h(\Psi,K_1)\vee h(\Psi,K_2);$$
\item[(P3)] For each $(\Psi_1,K_1)\in \mathcal{I}(\sX_1)$, $(\Psi_2,K_2)\in \mathcal{I}(\sX_2)$, then $(\Psi_1\X \Psi_2,K_1\X K_2)\in \mathcal{I}(\sX_1\X \sX_2)$ and
$$h(\Psi_1\X \Psi_2,K_1\X K_2)=h(\Psi,K_1)\wedge h(\Psi,K_2);$$
\item[(P4)] If the family of semiflows $\Psi_\lam$ depends on $\lam$ continuously and there exists an admissible set $\cN$ (w.r.t. $\Psi_\lam$) such that
$K_\lam=S_\8(\Psi_\lam,\cN)\ss \mb{int} \cN$ for all $\lam\in [0,1]$, then
$$h(\Psi_0,K_0)=h(\Psi_1,K_1)=h(\Psi_\lam,K_\lam),\hs \lam\in [0,1].$$
\eenu

\setcounter {equation}{0}
\section{Global invariant manifolds}

\hs\,\, In this section, we establish a nonautonomous global invariant manifold for the following abstract parabolic equation: 
\be\label{e3.1e}
\frac{\d u}{\d t}+A u=\lam u+f(u)+g(t)
\ee
on a Banach space $X$, where $A$ is a sectorial operator on $X$, $\lam\in\R$ is the bifurcation parameter, $f$ is a global Lipschitz continuous mapping from $X^\a$ ($0\leq \a<1$) to $X$ and is sublinear as $\|u\|_\a\ra \8$ uniformly on bounded $\lam$-intervals, and $g:\R\ra X$ is a bounded and H\"older continuous mapping.


\subsection{Spectrum decomposition}
\hs\,\, Denote $\|\.\|$ and $\|\.\|_\a$ the norms of the Banach spaces $X$ and $X^\a$ ($0\leq \a <1$), respectively.

Let $\sig(A)$ denote the spectrum of $A$ and
$${\rm Re}\sig(A):=\{{\rm Re} z:z\in \sig(A)\}.$$
Assume that the spectrum $\sigma(A)$ has a decomposition $\sig(A)=\sig_1\cup\sig_2\cup\sig_3$ such that
$$
  \sig_1=\sig(A)\cap\{{\rm Re}\lam\leq \beta_1\},\hs \sig_2=\{\mu\},\hs \sig_3=\sig(A)\cap\{{\rm Re}\lam\geq \beta_2\},
$$
where $\mu$ is an isolated eigenvalue of $A$ and $\beta_1,\beta_2$ are two real numbers.
Clearly,
$
  {\rm Re}\,\sig(A)\cap(\beta_1,\beta_2)=\{{\rm Re}\,\mu\}.
$
Moreover, the space $X$ also has a decomposition
$$
  X=X_1\oplus X_2 \oplus X_3,
$$
where $X_2$ is a finite-dimensional subspace of $X$. Let
$$
  P_i: X \ra X_i,\hs i\in\{1,2,3\}
  $$
be the projection from $X$ to $X_i$. Set
$$
  Y_i=X^\a\cap X_i,\Hs i\in\{1,2,3\}.
  $$
Then by the finite dimensionality of $X_2$, one finds that $Y_2$ coincides with $X_2$. Also, it holds that
$$
  X^\a=Y_1\oplus Y_2\oplus Y_3.
  $$

Pick a positive number $\beta$ satisfying
\be\label{beta}
  \beta<\min\{{\rm Re}\mu-\beta_1,\beta_2-{\rm Re}\mu\}.
\ee
Write
\be\label{JA}
  J=({\rm Re}\mu-\frac{1}{4}\beta,{\rm Re}\mu+\frac{1}{4}\beta),\hs A_\lam=A-\lam I, \,\lam\in J,\hs A^i=A_\lam|_{X_i}.
\ee
Thanks to the basic knowledge on sectorial operators (see e.g., \cite[Theorems 1.5.3, 1.5.4]{Hen}), one can immediately conclude that there exists $M>0$ (depending on $A$) such that for $\a\in[0,1)$,
\begin{align}
  &\|\Lam^\a {\e}^{-A^1 t}\|\leq M\e^{\frac{3}{4}\beta t},\hs \|\e^{-A^1 t}\|\leq M\e^{\frac{3}{4}\beta t},\Hs t\leq 0,\label{e3.2}\\
  &\|\Lam^\a \e^{-A^2 t}\|\leq M\e^{\frac{1}{4}\beta |t|},\hs \|\e^{-A^2 t}\|\leq M\e^{\frac{1}{4}\beta |t|},\Hs t\in \R,\label{e3.3}\\
  &\|\Lam^\a \e^{-A^3 t}P_3\Lam^{-\a}\|\leq M\e^{-\frac{3}{4}\beta t},\hs \|\Lam^\a \e^{-A^3 t}\|\leq Mt^{-\a}\e^{-\frac{3}{4}\beta t},\Hs t>0,\label{e3.4}
\end{align}
where $\Lam=A+aI$, and $a$ is a positive number with ${\rm Re} \,\sig(\Lam)>0$. 

\vs

\subsection{Mathematical setting}

\hs \,\, Let $C_b(\R,X)$ denote the set of bounded continuous functions from $\R$ to $X$, which is equipped with the compact-open topology generated by the metric $$\varrho(g_1,g_2)=\sum_{n=1}^\8\frac{1}{2^n}\.\frac{\max_{t\in[-n,n]}|g_1(t)-g_2(t)|}{1+\max_{t\in[-n,n]}|g_1(t)-g_2(t)|}.$$
Then $C_b(\R,X)$ is a complete metric space.
\vs
Assume that $g\in C_b(\R,X)$, where $g(t)=g(\.,t)$. Define the hull of $g$ as follows:
$$\cH:=\cH[g]=\ol{\{g(\tau+\.):\tau\in\R\}}^{C_b(\R,X)}.$$
In the following we always assume that $\cH$ is {\it compact}. The translation group $\theta$ on $\cH$ is defined by
$$
(\theta_tg)(\cdot)=g(t+\cdot),\Hs g\in \cH,\hs t\in \R.
$$
It is easy to see that $\theta=\{\theta_t\}_{t\in\R}$ is a base flow on the space $\cH$.
\vs
Then the equation \eqref{e3.1e} can be written as an abstract cocycle system in $X^\a$:
\be\label{e3.7e}
u_t+Au=\lam u+ f(u)+p(t), \Hs p\in \cH,
\ee
By the classical results (see \cite[Theorem 3.3.3]{Hen}), we conclude that the Cauchy problem of \eqref{e3.7e} is well-posed. Namely, for each $u_0\in X^\a$, $t_0\in \R$ and $p\in \cH$, there exists $T>t_0$, such that \eqref{e3.7e} has a unique solution  $u(t)=u(t,t_0;u_0,p)$ on $[t_0,T)$  with $u(t_0)=u_0$. By the assumptions on $f$ and $g$, we knows that the solution $u(t)$ globally exists for $t\geqslant t_0$, $u_0\in X^\a$, $p\in \cH$.
Set 
$$
 \psi(t,p)u_0=u(t,0;u_0,p),\Hs u_0\in X^\a,\hs p\in \cH,\hs t\geqslant 0.
 $$
Then $\psi$ is a cocycle semiflow on $X^\a$ driven by the translation group $\theta$ on the base space $\cH$.
\vs
Let the $\beta$ be given by \eqref{beta}. Define the space by
\be\label{aa}
  \cV_\beta=\big\{u\in \cC(\mathbb{R};X^\a):\sup\limits_{t\in \mathbb{R}}\e^{-\frac{\beta}{2} |t|}\|u(t)\|_\a<\infty\big\},
\ee
with the norm as follows:
$$
  \|u\|_{\cV_\beta}=\sup\limits_{t\in \mathbb{R}}\e^{-\frac{\beta}{2} |t|}\|u(t)\|_\a<\infty.
$$
Then $\cV_\beta$ is a Banach space.
\vs

\subsection{Existence of global invariant manifolds}
\hs \,\, In this subsection, we establish the existence of nonautonomous global invariant manifolds for \eqref{e3.7e}, which extends the result in \cite[Theorem 4.1]{LW21} to the case of nonautonomous systems. For this purpose, we assume that the following condition on $f$ holds true.
\vs
\noindent($\mathbf{F}$) The Lipschitz constant $L_f$ of $f$ satisfies
$$
   L_f M \int^\infty_0(2+s^{-\a})\e^{-\frac{1}{4}\beta s}\d s<1.
  $$
Let
$$
 Y_{i,j}=Y_i\oplus Y_j,
$$
where $i,j=1,2,3,i\neq j$. Then we can obtain the following conclusion.

\begin{thm}\label{t3.1}
Let the assumption {\rm(}$\mathbf{F}${\rm)} hold. Then for each $\lam\in J$ and each $p\in\cH$, there exists a Lipschitz continuous mapping $$\xi_{\lam,p}:Y_2\ra Y_{13},$$  such that the system \eqref{e3.7e} has a global invariant manifold $\cM_{\lam,p}$, defined by
$$
  \cM_{\lam,p}=\{w+\xi_{\lam,p}(w):w\in Y_2\},\Hs \lam\in J,\,\,p\in\cH.
$$
\end{thm}

\begin{proof}

In the proof of this theorem, we develop the technique from \cite{LW21}. Since the system \eqref{e3.7e} is nonautonomous, we slightly modify the proof of \cite[Theorem 3.1]{LW21} and further complement the proof of the invariance of the nonautonomous manifold. For the readers' convenience, we give the proof in details.

Let us consider the Banach space $\cV_{\beta}$ given by \eqref{aa}. Assume that $\lam\in J$. For each fixed $p\in\cH$ and $w\in Y_2$, define a mapping $G:=G_{\lam,p,w}$ on $\cV_{\beta}$ by
\begin{align}\label{G}
G(u)(t)=&\e^{-A^2 t}w+\!\int^t_0\e^{-A^2(t-s)}P_2[f(u(s))+p(s)]{\rm d}s \no \\
&+\!\int^t_{-\infty}\e^{-A^3(t-s)}P_3[f(u(s))+p(s)]{\rm d}s\no \\
 &-\int^\infty_t\e^{-A^1(t-s)}P_1[f(u(s))+p(s)]{\rm d}s.
\end{align}
In what follows we prove that $G$ is a contraction mapping on $\cV_{\beta}$. Firstly, we show $G:\cV_{\beta}\ra \cV_{\beta}$.
Let $u\in \cV_{\beta}.$ Observing that $f$ is globally Lipschitz continuous and that $g$ is bounded, there is a constant $C_1>0$ (depending on $f$ and $g$) such that
$$
  \|f(u)+p(s)\|\leq C_1(\|u\|_\a+2),\Hs \A \,u\in X^\a.
$$
From \eqref{e3.2}-\eqref{e3.4} and \eqref{G}, we easily know that
\begin{align}\label{e3.11}
\|G(u)\|_\a\leq& M\e^{\frac{1}{4}\beta|t|}\|w\|_\a+M\int^{t_2}_{t_1}\e^{\frac{1}{4}\beta|t-s|}C_1(\|u(s)\|_\a+2){\rm d}s\nonumber\\
&+M\int_{-\infty}^t(t-s)^{-\a} \e^{-\frac{3}{4}\beta(t-s)}C_1(\|u(s)\|_\a+2){\rm d}s\nonumber\\
&+M\int^\infty_t\e^{\frac{3}{4}\beta(t-s)}C_1(\|u(s)\|_\a+2){\rm d}s,
\end{align}
where
$$
  t_1:=\min\{t,0\},\hs t_2:=\max\{t,0\}.
  $$
Now let us estimate the terms on the right-hand side of the inequality \eqref{e3.11} in $\cV_{\beta}$ one by one.
For $t\geq 0,$
\begin{align}\label{e3.12}
 &\e^{-\frac{\beta}{2}|t|}M\int^{t_2}_{t_1}\e^{\frac{1}{4}\beta|t-s|}C_1(\|u(s)\|_\a+2){\rm d}s\nonumber \\
 =&\e^{-\frac{\beta}{2}t}M\int^{t}_{0}\e^{\frac{1}{4}\beta(t-s)}C_1(\|u(s)\|_\a+2){\rm d}s\nonumber\\
=&MC_1\int^{t}_{0}\e^{-\frac{1}{4}\beta(t-s)}\e^{{-\frac{\beta}{2}} s}(\|u(s)\|_\a+2){\rm d}s\nonumber\\
\leq& MC_1\int^{t}_{0}\e^{-\frac{1}{4}\beta(t-s)}\(\e^{-\frac{\beta}{2}s}\|u(s)\|_\a+2\){\rm d}s.
\end{align}
Similarly, if $t\leq 0,$
\begin{align}\label{e3.13}
 & \e^{-\frac{\beta}{2}|t|}M\int^{t_2}_{t_1}\e^{\frac{1}{4}\beta|t-s|}C_1(\|u(s)\|_\a+2){\rm d}s\nonumber \\
 =&\e^{\frac{\beta}{2} t}M\int^{0}_{t}\e^{\frac{1}{4}\beta(s-t)}C_1(\|u(s)\|_\a+2){\rm d}s\nonumber\\
=&MC_1\int^{0}_{t}\e^{-\frac{1}{4}\beta(s-t)}\e^{\frac{\beta}{2} s}(\|u(s)\|_\a+2){\rm d}s\nonumber\\
\leq& MC_1\int^{0}_{t}\e^{-\frac{1}{4}\beta(s-t)}\(\e^{\frac{\beta}{2}s}\|u(s)\|_\a+2\){\rm d}s.
\end{align}
Taking \eqref{e3.12} and \eqref{e3.13} into account, it yields that
\begin{align}\label{e3.14}
 &\e^{-\frac{\beta}{2}|t|}M\int^{t_2}_{t_1}\e^{\frac{1}{4}\beta|t-s|}C_1(\|u(s)\|_\a+2){\rm d}s\nonumber\\
\leq& M C_1\int^{t_2}_{t_1}\e^{-\frac{1}{4}\beta|t-s|}\big(\e^{-\frac{\beta}{2}|s|}\|u(s)\|_\a+2\big){\rm d}s\nonumber \\
\leq &M C_1\int^{\8}_{0}\e^{-\frac{1}{4}\beta s}{\rm d}s\cdot(\|u\|_{\cV_{\beta}}+2), \Hs t\in \R.
\end{align}
Since
\begin{align*}
 \e^{-\frac{\beta}{2}|t|}=\e^{-\frac{\beta}{2}|t-s+s|}\leq \e^{-\frac{\beta}{2}|s|}\e^{\frac{\beta}{2}|t-s|},\hs t\in \R,
\end{align*}
we easily deduce that
\begin{align}\label{e3.15}
&\e^{-\frac{\beta}{2}|t|}M\int_{-\infty}^t(t-s)^{-\a} \e^{-\frac{3}{4}\beta(t-s)}C_1(\|u(s)\|_\a+2)\d s\nonumber \\
\leq&
M C_1\int_{-\infty}^t(t-s)^{-\a} \e^{\frac{\beta}{2}|t-s|}\e^{-\frac{3}{4}\beta(t-s)}\big[\e^{-\frac{\beta}{2}|s|}(\|u(s)\|_\a+2)\big]\d s\nonumber \\
\leq&M C_1\int_{-\infty}^t(t-s)^{-\a} \e^{-\frac{1}{4}\beta(t-s)}\(\e^{-\frac{\beta}{2}|s|}\|u(s)\|_\a+2\)\d s\nonumber \\
\leq &M C_1\int^\infty_0s^{-\a}\e^{-\frac{1}{4}\beta s}\d s\cdot(\|u\|_{\cV_{\beta}}+2),
\end{align}
and
\begin{align}\label{e3.16}
&\e^{-\frac{\beta}{2}|t|}M\int^\infty_t\e^{\frac{3}{4}\beta(t-s)}C_1(\|u(s)\|_\a+2)\d s \nonumber \\
\leq &M C_1\int^\infty_t\e^{\frac{\beta}{2}|t-s|}\e^{\frac{3}{4}\beta(t-s)}\e^{-\frac{\beta}{2}|s|}\(\|u(s)\|_\a+2\)\d s \nonumber \\
\leq &M C_1\int^\infty_t\e^{\frac{1}{4}\beta(t-s)}\(\e^{-\frac{\beta}{2}|s|}\|u(s)\|_\a+2\)\d s\nonumber \\
\leq &M C_1\int^\infty_0\e^{-\frac{1}{4}\beta s}{\d} s\cdot(\|u\|_{\cV_{\beta}}+2).
\end{align}
Therefore one immediately concludes from \eqref{e3.11}, \eqref{e3.14}, \eqref{e3.15} and \eqref{e3.16} that
\begin{align*}
\e^{-\frac{\beta}{2}|t|}\|G(u)\|_\a
\leq& M\e^{-\frac{1}{4}\beta|t|}\|w\|_\a+\e^{-\frac{\beta}{2}|t|}M\int^{t_2}_{t_1}\e^{\frac{1}{4}\beta|t-s|}C_1(\|u(s)\|_\a+2){\rm d}s\nonumber\\
&+\e^{-\frac{\beta}{2}|t|}M\int_{-\infty}^t(t-s)^{-\a} \e^{-\frac{3}{4}\beta(t-s)}C_1(\|u(s)\|_\a+2){\rm d}s\nonumber\\
&+\e^{-\frac{\beta}{2}|t|}M\int^\infty_t\e^{\frac{3}{4}\beta(t-s)}C_1(\|u(s)\|_\a+2){\rm d}s\nonumber \\
\leq&M\|w\|_\a+M C_1\int^{\8}_{0}\e^{-\frac{1}{4}\beta s}{\rm d}s\cdot(\|u\|_{\cV_{\beta}}+2)\nonumber\\
&+M C_1\int^\infty_0s^{-\a}\e^{-\frac{1}{4}\beta s}\d s\cdot(\|u\|_{\cV_{\beta}}+2)\nonumber \\
&+M C_1\int^\infty_0\e^{-\frac{1}{4}\beta s}{\d} s\cdot(\|u\|_{\cV_{\beta}}+2)\nonumber \\
\leq& M\|w\|_\a\!+\!MC_1\!\int^\infty_0\!(2\!+\!s^{-\a})\e^{-\frac{1}{4}\beta s}\d s\! \cdot\!\(\|u\|_{\cV_{\beta}}\!+\!2\)<\infty,\,\, \forall t\in \mathbb{R}.
\end{align*}
This implies that $\|Gu\|_{\cV_{\beta}}<\infty,$ and hence $Gu\in \cV_{\beta}$.
\vs
Secondly, let us further verify that $G$ is a contraction mapping on $\cV_{\beta}.$ Assume that $u_1,u_2\in \cV_{\beta}$. Then
\begin{align}\label{e3.17}
\e^{-\frac{\beta}{2}|t|}\|G(u_1)-G(u_2)\|_\a
&\leq \e^{-\frac{\beta}{2}|t|}\|\int^{t}_{0}\e^{-A^2(t-s)}P_2\big(f(u_1)-f(u_2)\big)\d s\|_\a\nonumber\\
&\, +\e^{-\frac{\beta}{2}|t|}\|\int^t_{-\infty}\!\e^{-A^3(t-s)}P_3\big(f(u_1)-f(u_2)\big){\rm d}s\|_\a\nonumber \\
&\, +\e^{-\frac{\beta}{2}|t|}\|\int^\infty_t\!\e^{-A^1(t-s)}P_1\big(f(u_1)-f(u_2)\big)\d s\|_\a.
\end{align}
Using the same argument as the derivation of \eqref{e3.14}, one can easily obtain that
\begin{align}\label{e3.18}
&\e^{-\frac{\beta}{2}|t|}\|\int^{t}_{0}\e^{-A^2(t-s)}P_2\big(f(u_1)-f(u_2)\big)\d s\|_\a\nonumber\\
\leq& \,ML_f\int^{t_2}_{t_1}\e^{-\frac{1}{4}\beta|t-s|}\big(\e^{-\frac{\beta}{2}|s|}\|u_1-u_2\|_\a\big)\d s,
\end{align}
where $L_f$ denotes the Lipschitz constant of $f$. Moreover, we can find that
\begin{align}
&\e^{-\frac{\beta}{2}|t|}\|\int^t_{-\infty}\e^{-A^3(t-s)}P_3\big(f(u_1)-f(u_2)\big){\rm d}s\|_\a\nonumber \\
\leq &\,ML_f\int^t_{-\infty}(t-s)^{-\a}\e^{-\frac{1}{4}\beta(t-s)}\big(\e^{-\frac{\beta}{2}|s|}\|u_1-u_2\|_\a\big)\d s,\label{e3.19}\\
&\e^{-\frac{\beta}{2}|t|}\|\int^\infty_t\e^{-A^1(t-s)}P_1(f(u_1)-f(u_2))\d s\|_\a\nonumber \\
\leq& \, ML_f\int_t^{\infty}\e^{\frac{1}{4}\beta(t-s)}\big(\e^{-\frac{\beta}{2}|s|}\|u_1-u_2\|_\a\big)\d s.\label{e3.20}
\end{align}
Combining \eqref{e3.17}-\eqref{e3.20}, it yields that
\begin{align}\label{e3.21}
\e^{-\frac{\beta}{2}|t|}\|Gu_1-Gu_2\|_\a&\leq ML_f\int^{t_2}_{t_1}\e^{-\frac{1}{4}\beta|t-s|}\big(\e^{-\frac{\beta}{2}|s|}\|u_1-u_2\|_\a\big)\d s\nonumber\\
&+ML_f\int^t_{-\infty}(t-s)^{-\a}\e^{-\frac{1}{4}\beta(t-s)}\!\big(\e^{-\frac{\beta}{2}|s|}\|u_1-u_2\|_\a\big)\d s\nonumber\\
&+ML_f\int_t^{\infty}\e^{\frac{1}{4}\beta(t-s)}\big(\e^{-\frac{\beta}{2}|s|}\|u_1-u_2\|_\a\big)\d s\nonumber\\
&\leq M L_f \int^\infty_0(2+s^{-\a})\e^{-\frac{1}{4}\beta s}\d s\cdot \|u_1-u_2\|_{\cV_{\beta}}.
\end{align}
That is,
$$\|Gu_1-Gu_2\|_{\cV_{\beta}}\leq M L_f \int^\infty_0(2+s^{-\a})\e^{-\frac{1}{4}\beta s}\d s\.\|u_1-u_2\|_{\cV_{\beta}}.$$
By the condition ({\bf F}), one knows that $G$ is a contraction mapping on $\cV_{\beta}$.
According to the Banach contraction mapping principle, we deduce that $G$ has a fixed point $u_w:=u_{\lam,p,w}\in \cV_{\beta},$ which 
satisfies $P_2u_w(0)=w$ and
\begin{align}\label{e3.22}
   u_w(t)=&\e^{-A^2 t}w+\!\int^t_0\e^{-A^2(t-s)}P_2[f(u_w(s))+p(s)]\d s \no \\
   &+\!\int^t_{-\infty}\e^{-A^3(t-s)}P_3[f(u_w(s))+p(s)]\d s\nonumber\\
 &-\int^\infty_t\e^{-A^1(t-s)}P_1[f(u_w(s))+p(s)]\d s.
\end{align}

In the following we construct the invariant manifold mapping $\xi_{\lam,p}$ and show that $\xi_{\lam,p}$ is Lipschitz continuous. For this goal, we set
$$
\Gam(w)=u_w(0),\hs w\in Y_2,
$$
and prove that $\Gam(w)$ is Lipschitz continuous for $w\in Y_2$. Assume that $w_1,w_2\in Y_2$. Then using the same estimation of \eqref{e3.17}, we easily deduce by \eqref{e3.22} that
\begin{align*}
  &\e^{-\frac{\beta}{2}|t|}\|u_{w_1}(t)-u_{w_2}(t)\|_\a\\
  \leq&\, M\e^{-\frac{1}{4}\beta|t|}\|w_1-w_2\|_\a +ML_f\int^{t_2}_{t_1}\e^{-\frac{1}{4}\beta|t-s|}\big(\e^{-\frac{\beta}{2}|s|}\|u_{w_1}-u_{w_2}\|_\a\big) \d s\\
&+ML_f\int^t_{-\infty}(t-s)^{-\a}\e^{-\frac{1}{4}\beta(t-s)}\big(\e^{-\frac{\beta}{2}|s|}\|u_{w_1}-u_{w_2}\|_\a\big)\d s\\
&+ML_f\int_t^{\infty}\e^{\frac{1}{4}\beta(t-s)}\big(\e^{-\frac{\beta}{2}|s|}\|u_{w_1}-u_{w_2}\|_\a\big)\d s\\
\leq&\, M\|w_1-w_2\|_\a+ M L_f\int^\infty_0(2+s^{-\a})\e^{-\frac{1}{4}\beta s}\d s\cdot \|u_{w_1}-u_{w_2}\|_{\cV_{\beta}},\hs t\in \R,
\end{align*}
which implies that
$$\|u_{w_1}-u_{w_2}\|_{\cV_{\beta}}\leq M\|w_1-w_2\|_\a+L_\beta L_f\|u_{w_1}-u_{w_2}\|_{\cV_{\beta}},$$
where $L_\beta=M \int^\infty_0(2+s^{-\a})\e^{-\frac{1}{4}\beta s}\d s$.
Thereby,
\begin{align*}
\|\Gam(w_1)-\Gam(w_2)\|_\a&=\|u_{w_1}(0)-u_{w_2}(0)\|_\a\leq \|u_{w_1}(t)-u_{w_2}(t)\|_{\cV_{\beta}}\\[1ex]
&\leq M(1-L_\beta L_f)^{-1}\|w_1-w_2\|_\a,
\end{align*}
and hence $\Gam(w)$ is Lipschitz continuous in $w$.

Now, we use \eqref{e3.22} to define a mapping $\xi_{\lam,p}:Y_2\ra Y_{13}$ as
\begin{align}\label{e3.23}
\xi_{\lam,p}(w)=&\int^0_{-\infty}\e^{A^3s}P_3[f(u_w(s))+p(s)]\d s \no \\
 &-\int^\infty_0\e^{A^1s}P_1[f(u_w(s))+p(s)]\d s, \hs w\in Y_2.
\end{align}
It is easy to see from \eqref{e3.22} that
\be\label{e3.24}
\Gam(w)=u_w(0)=w+\xi_{\lam,p}(w),\hs w\in Y_2.
\ee
Recalling that $\Gam(w)$ is Lipschitz continuous in $w\in Y_2$, we immediately conclude from \eqref{e3.24} that $\xi_{\lam,p}:Y_2\ra Y_{13}$ is Lipschitz continuous as well.
\vs
Finally, define
$$
  \cM_{\lam,p}=\{w+\xi_{\lam,p}(w):w\in Y_2\},\Hs \lam\in J,\hs p\in\cH.
  $$
Therefore, $\cM_{\lam,p}$ is the desired global invariant manifold for the skew-product flow generated by \eqref{e3.7e}.
\vs
It remains only to show the invariance of $\cM_{\lam,p}$. Let $(w_0,v_0)\in \cM_{\lam,p}$, where $v_0=\xi_{\lam,p}(w_0)$. By the definition, we need to prove that there exists a full solution $u(t)=(w(t),v(t)),t\in \R$ such that
$$(w(t),v(t))\in\cM_{\lam,p}\hs \mb{for} \hs t\in\R.$$
Let $w(t)$ with $w(0)=w_0$ denote the full solution of the following finite-dimensional system:
$$
w_t+A^2w=P_2\(f(w+\xi_{\lam,p}(w)\)+P_2p(t), \Hs p\in \cH,
$$
where $w=P_2u$, $A^2=P_2A_\lam=P_2(A-\lam I)$, $\lam\in J$. Set
$$v(t)=\xi_{\lam,p}(w(t)),\Hs t\in\R.$$
Then one can see that this defines a curve $(w(t),v(t))\in \cM_{\lam,p}$ for $t\in\R$  through the point $(w_0,v_0)$. By the construction of $\xi_{\lam,p}$, we easily find that $u(t)=w(t)+v(t)$ satisfies \eqref{e3.22}. Thus, it suffices to prove that $u(t)$ is a full solution of \eqref{e3.7e}.

Indeed,
assume that $\~u\in \cV_\beta$ is a full solution of the equation:
$$
u_t+Au=\lam u+ f(w+\xi_{\lam,p}(w))+p(t), \Hs p\in \cH,
$$
with $\~u(0)=w_0+v_0$. We rewrite $\~u(t)$ as
\be\label{e3.24e}
  \~u(t)=u^1(t)+u^2(t)+u^3(t),
\ee
where $u^i(t)=P_i\~u(t),i\in\{1,2,3\}.$ Moreover, $u^i(t)$ can be expressed, respectively, as the integral form:
\begin{align}
  &u^2(t)=\e^{-A^2 t}u^2(0)\!+\!\int^t_{0}\e^{-A^2(t-s)}P_2[f(w(s)+\xi_{\lam,p}(w(s)))+p(s)]{\rm d}s,\hs t\in \mathbb{R},\label{e3.8e}\\
  &u^1(t)=\e^{-A^1 (t-t_0)}u^1(t_0)\!+\!\int^t_{t_0}\e^{-A^1(t-s)}P_1[f(w(s)+\xi_{\lam,p}(w(s))+p(s)]{\rm d}s,\,\, t\leq t_0,\label{e3.9e}\\
  &u^3(t)=\e^{-A^3 (t-t_0)}u^3(t_0)\!+\!\int^t_{t_0}\e^{-A^3(t-s)}P_3[f(w(s)+\xi_{\lam,p}(w(s))+p(s)]{\rm d}s,\,\, t\geq t_0.\label{e3.10e}
\end{align}
From \eqref{e3.2}, we infer that for $t\leq t_0$,
\begin{align*}
  \|\e^{-A^1 (t-t_0)}u^1(t_0)\|_\a
\leq& \,M\e^{\frac{3}{4}\beta(t-t_0)}\|\~u(t_0)\|_\a\\
=& \,M\e^{\frac{3}{4}\beta t}\e^{-\frac{1}{4}\beta t_0}\(\e^{-\frac{\beta}{2} t_0}\|\~u(t_0)\|_\a\)\\
\leq &\,M\e^{\frac{3}{4}\beta t}\e^{-\frac{1}{4}\beta t_0}\|\~u\|_{\cV_\beta}\ra 0\hs \mb{as}\hs t_0\ra \infty.
\end{align*}
Similarly, one can find by \eqref{e3.4} that for $t\geq t_0$,
\begin{align*}
  \|\e^{-A^3 (t-t_0)}u^3(t_0)\|_\a
\leq& \,M\e^{-\frac{3}{4}\beta(t-t_0)}\|u(t_0)\|_\a\\
=& M\,\e^{-\frac{3}{4}\beta t}\e^{\frac{1}{4}\beta t_0}(\e^{\frac{\beta}{2} t_0}\|u(t_0)\|_\a)\\
\leq &\,M\e^{-\frac{3}{4}\beta t}\e^{\frac{1}{4}\beta t_0}\|u\|_{\cV_\beta}\ra 0\hs \mb{as}\hs t_0\ra -\infty.
\end{align*}
Passing to the limit as $t_0\ra \infty$ and $t_0\ra -\infty$ in \eqref{e3.9e}, \eqref{e3.10e}, respectively, we can immediately conclude from \eqref{e3.24e}-\eqref{e3.10e} that $\~u(t)$ satisfies \eqref{e3.22} as well. Therefore, we infer from the uniqueness of solutions of \eqref{e3.7e} that $u(t)=\~u(t)$ is indeed a solution of \eqref{e3.7e}.
\end{proof}

\begin{remark}\label{r3.1}
  It is worth mentioning that if $f$ and $g$ are bounded, then one can easily deduce by \eqref{e3.23} that there exists $M_1>0$ (depending on the boundedness of $f$ and $g$) such that
$$
  \|\xi_{\lam,p}(w)\|_\a\leq MM_1\int_0^{\infty}(1+s^{-\a})\e^{-\frac{3}{4}\beta s}\d s, \hs w\in Y_2,
$$
which shows that $\xi_{\lam,p}$ is bounded.
\end{remark}

\setcounter {equation}{0}
\section{Homotopy and bifurcation of the reduced equation}

\hs \,In this section we investigate the dynamic bifurcation from infinity of the reduced system.  By Theorem \ref{t3.1}, we restrict the original equation \eqref{e3.7e} to the global invariant manifold $\cM_{\lam,p}$, which can generate a nonautonomous finite-dimensional system:
\be\label{e4.1e}
w_t+A^2 w=P_2f(w+\xi_{\lam,p}(w))+P_2p(t),\hs w\in Y_2,
\ee
where $w=P_2u$, $A^2=P_2A_\lam=P_2(A-\lam I)$, $\lam\in J=(\mu-\frac{1}{4}\beta,\mu+\frac{1}{4}\beta).$
\vs
Let $\phi_\lam$ denote a cocycle semiflow on $Y_2$ driven by the translation group $\theta$ on the compact base space $\cH$. Define the skew-product flow $\Phi_\lam$ by
$$
  \Phi_\lam(t)(w,p)=(\phi_\lam(t,p)w,\theta_tp),\Hs \A(w,p)\in Y_2\X \cH,\hs t\geqslant 0.
  $$
Then $\Phi_\lam$ is a global semiflow associated with the reduced equation \eqref{e4.1e}.

\subsection{Homotopy and bifurcation}
\hs \, In order to compute the Conley index of the invariant sets for \eqref{e4.1e}, we construct a homotopy between the reduced equation \eqref{e4.1e} and a product flow. To the end, let us consider the following parameterized equation:

\be\label{e4.2e}
w_t+A^2 w=P_2f(w+\xi_{\lam,\nu p}(w))+P_2\nu p(t),\hs w\in Y_2,
\ee
where $\nu\in [0,1]$ is the homotopy parameter. Denote $\Phi_\lam^\nu$ the skew-product flow generated by \eqref{e4.2e} on the product space $Y_2\X \cH$.

\vs
Note that if $\nu=0$, then the equation \eqref{e4.2e} can be transformed into the finite-dimensional autonomous system:
\be\label{e4.3e}
w_t+A^2 w=P_2f(w+\xi_{\lam,0}(w)),\hs w\in Y_2.
\ee
where $\xi_{\lam,0}$ denotes the invariant manifold mapping defined by \eqref{e3.23} with $p=0$. Let $\vp_\lam$ denote the semiflow generated by \eqref{e4.3e} and let
$\Pi_i \,(i=1,2)$ be the projections from $Y_2\X \cH$ to $Y_2$ and $\cH$, respectively, given by
$$\Pi_1(w,p)=w,\hs \Pi_2(w,p)=p,\Hs (w,p)\in Y_2\X \cH.$$
Then we have the following results on invariant sets for the skew-product flow and the product flow.
\bl\label{l4.1}
Assume that there exists a bounded isolating neighborhood $\cN\ss Y_2\X \cH$ of the skew-product flow $\Phi^\nu_\lam$ for all $\nu\in [0,1]$ and write $K_\nu=S_\infty(\Phi^\nu_\lam,\cN)$. Then the following assertions hold.
\benu
\item[{\rm (1)}] $\Phi^0_\lam=\vp_\lam\times \theta$;
\item[{\rm (2)}] If $\cN=\Pi_1\cN\times \Pi_2 \cN$, then $\Pi_i\cN\,(i=1,2)$ is a bounded isolating neighborhood for $\vp_\lam$ and $\theta$, respectively. Moreover, \be\label{3.8}
    K_0=S_\infty(\vp_\lam,\Pi_1\cN)\times S_\infty(\theta,\Pi_2\cN)=\Pi_1K_0\times \Pi_2K_0;
    \ee
\item[{\rm (3)}] $h(\Phi_\lam^\nu,K_\nu)=h(\Phi^0_\lam,K_0)=h(\vp_\lam,\Pi_1K_0)\wedge h(\theta,\Pi_2K_0),\hs \A \nu\in [0,1].$
\eenu
\el
\begin{proof}
(1) By the definition of $\Phi^0_\lam$, we know that
\begin{align*}
\Phi^0_\lam(t)(w_0,p)=\big(w(t,0;w_0,0),\theta_t p\big),\hs (w_0,p)\in Y_2\times \cH.
\end{align*}
Since $w(t,0;w_0,0)$ is a solution of the autonomous equation \eqref{e4.3e}, one has $w(t,0;w_0,0)=\vp_\lam(t)w_0$. Thus
$$\Phi^0_\lam(t)(w_0,p)=(\vp(t)w_0,\theta_t p),$$
which shows the validity of (1).

\vs
(2) Let $(w_0,p)\in S_\infty(\vp_\lam,\Pi_1 \cN)\times S_\infty(\theta,\Pi_2 \cN)$. Then
$$
  \vp_\lam(t)w_0\in \Pi_1 \cN \hs {\rm and} \hs \theta_tp\in \Pi_2 \cN \Hs \mb {for all}\hs t\in \R.
  $$
 From the assertion (1), one infers that
 $$
   \Phi^0_\lam(t)(w_0,p)=(\vp_\lam(t)w_0,\theta_t p)\in \Pi_1\cN\times \Pi_2 \cN=\cN,\hs \A t\in \R.$$
   Hence,
 $$
   S_\infty(\vp_\lam,\Pi_1 \cN)\times S_\infty(\theta,\Pi_2 \cN)\subset K_0.
   $$
Conversely, if $(w_0,p)\in K_0$, then
$$
  (\vp_\lam(t)w_0,\theta_t p)\in \cN \Hs {\rm for \,\,all} \hs t\in \R.
$$
Noticing that $\cN=\Pi_1\cN\times \Pi_2 \cN$, we find that
$$
  \vp_\lam(t)w_0\in \Pi_1\cN \hs {\rm and}\hs  \theta_tp\in \Pi_2\cN\Hs \mb {for all}\hs  t\in \R,
  $$
which shows
$$
  w_0\in S_\infty(\vp_\lam,\Pi_1\cN)\hs \mb{and}\hs p\in S_\infty(\theta, \Pi_2\cN).
  $$
Thus $(w_0,p)\in S_\infty(\vp_\lam,\Pi_1 \cN)\times S_\infty(\theta,\Pi_2 \cN)$. Therefore, the equality \eqref{3.8} holds.
\vs
Note that $\cN=\Pi_1\cN\times \Pi_2 \cN$ is a bounded isolating neighborhood of $K_0=S_\infty(\vp_\lam,\Pi_1 \cN)\times S_\infty(\theta,\Pi_2 \cN)$ for $\Phi_\lam^0$. By the assertion (1), we conclude that $\Pi_i\cN\,(i=1,2)$ is a bounded isolating neighborhood for $\vp_\lam$ and $\theta$, respectively.
\vs
(3) By the continuation property of the Conley index (see (P4) or \cite[Chapter I, Theorem 12.2]{Ry}), we obtain that
$$
  h(\Phi^\nu_\lam,K_\nu)=h(\Phi^0_\lam,K_0),\hs \nu\in [0,1].
  $$
Hence it follows from (P3) (or \cite[Theorem 10.6]{Ry}) that the conclusion (3) holds true.
\end{proof}
\br
In the above lemma, we borrow the technique from \cite{Ward} to construct a homotopy between the nonautonomous reduced system and a product flow. Indeed, Ward \cite{Ward} studied the homotopy for nonautonomous ordinary differential equations by using the Conley index theory \cite{Ry}.
\er

Next, we establish our main results in this subsection on the dynamic bifurcation from infinity of the nonautonomous reduced equation \eqref{e4.1e}. In order to distinguish the Conley index of invatiant sets, we further make the following topological conditions on the base space $\cH$.
\vs
\noindent
({\bf H}) The space $\cH$ is compact and the singular cohomological group $H^k(\cH)$ is {\it free and finitely generated for each $k$}. 
\br\label{r4.1}
In the above condition, we assume that $H^k(\cH)$ is free and finitely generated for each $k$, which describes the topological structure of $\cH$. This topological condition  is very natural, and can be fulfilled in most cases.
For example, 
if $g$ is periodic in $t$, then $\cH\simeq \mathbb{S}^1${\rm (}one-dimensional sphere{\rm )}, ensuring this assumption.
\er
\bd\label{defn2.1}
The equation \eqref{e4.1e} is said to bifurcate from infinity at $\lam=\mu$ {\rm (}or, $(\infty,\mu)$ is a bifurcation point{\rm )}, if there are a sequence $\{\lam_n\}$ with $\lam_n\ra \lam_0$ {\rm(}as $n\ra \8${\rm)} and a sequence of bounded full solutions $w_{\lam_n}=w_{\lam_n}(t,p)$ of \eqref{e4.1e} such that
$$
  \|w_{\lam_n}\|_\infty\ra \8 \hs \mb{as}\hs n\ra\8,
$$
where $\|w_\lam\|_\8:=\sup_{t\in\R}\|w_\lam(t)\|$.
\ed

For each interval $I\subset \R$ and $\cN\subset Y_2\X\cH,$ set
$$\ba{ll}
  \cS(I,\cN)=\ol{\bigcup_{\lam\in I}(S_\infty(\Phi_\lam,\cN)\times \{\lam\})}.\ea
$$
If $\cN=Y_2\X\cH$, we simply write $\cS(I,\cN)=\cS(I)$. Let $\Gam\ss Y_2\X\cH\X \R$ and $\lam\in \R$. For convenience, write $$
\Gam[\lam]=\{\~w:\,\,(\~w,\lam)\in Y_2\X\cH\X \R\}.
$$
$\Gam[\lam]$ is called the {\em $\lam$-section of $\Gam$.}

\bt\label{t4.1}
Assume the conditions {\rm({\bf F})} and {\rm({\bf H})} hold. Let $\mu\in \sig_2$. Then  $(\infty,\mu)$ is a bifurcation point of \eqref{e4.1e}. Specifically,
there exist a sequence $\lam_n$ with $\lam_n\ra \mu$ {\rm($n\ra \8$)} and a sequence of bounded full solutions $w_{n}=w_{\lam_n}(t,0;w_0,p)$ such that
$$
  \|w_{n}\|_\8\ra \8\hs \mb{as}\hs n\ra \8.
  $$
\et
\bo
Let $\mu\in \sig_p$. Consider the parameterized equation \eqref{e4.2e} for $\lam\in J=(\mu-\frac{1}{4}\beta,\mu+\frac{1}{4}\beta)$.
Let us start with the linear equation:
\be\label{e4.5e}
w_t+A^2 w=0,\hs w\in Y_2,
\ee
where $w=P_2u$, $A^2=P_2A_\lam=P_2(A-\lam I)$, $\lam\in J$. Take two numbers $\ve_1$ and $\ve_2$ such that $\ve_1\in(\mu-\frac{1}{4}\beta,\mu)$ and $\ve_2\in (\mu,\mu+\frac{1}{4}\beta)$, respectively. Then if $\lam=\ve_1,\ve_2$, one can trivially see that the trivial solution set $\{0\}$ is an isolated invariant set for the semiflow $\pi_\lam$ generated by \eqref{e4.5e} in $Y_2$. From \cite[Chapter I, Corollary 11.2]{Ry}, we infer that there is a positive
integer $q$ such that
\be\label{e4.6e}
h(\pi_{\ve_1},\{0\})=\Sig^0, \hs h(\pi_{\ve_2},\{0\})=\Sig^q.\ee 

Next we study the parameterized equation \eqref{e4.2e}.
In what follows we use some techniques from \cite{Ry} (see also the proof of Theorem 3.2 in \cite{W2}) to prove that for every fixed $\ve>0$ with $$\ve_1<\mu-\ve<\mu+\ve<\ve_2,$$ there exists $R_\ve>0$ such that if $\lam\in[\ve_1,\mu-\ve]\cup[\mu+\ve,\ve_2]$, $\nu\in[0,1]$ and $w(t)$ is a bounded full solution of \eqref{e4.2e}, then
\be\label{e4.7e}
\|w(t)\|_\infty<R_\ve,
\ee
uniformly with respect to $p\in\cH$.
We verify by contraction and suppose the contrary. Then there would be $\ve_0>0$ and two sequences $\lam_n\in[\ve_1,\mu-\ve_0]\cup[\mu_0+\ve_0,\ve_2]$, $\nu_n\in [0,1]$, and a sequence of bounded full solutions $w_n(t)$ of the equation
\be\label{e3.8e}
w_t+A^2 w=P_2f(w+\xi_{\lam_n,\nu_np}(w))+P_2\nu_np(t),
\ee
such that $$\|w_n\|_\infty\ra \infty\hs \mb{as}\hs n\ra \infty.$$
It can be assumed that $\lam_n\ra \lam_1\in [\ve_1,\mu-\ve_0]\cup[\mu+\ve_0,\ve_2]$.
Set $v_n=w_n/\|w_n\|_\infty.$ Then $\|v_n\|_\infty=1$ for all $n$.
By some quite standard arguments, one can show that there is a subsequence of $v_n$, still denoted by $v_n$, such that
$v_n$ converges uniformly on every compact interval of $\R$ to a function $v=v(t)$. Note that $v_n$ satisfies
\be\label{e4.9e}
\frac{{\rm d}v_n}{{\rm d}t}+A^2v_n=P_2f(w_n+\xi_{\lam_n,\nu_np}(w_n))/\|w_n\|_\infty+P_2\nu_np(t)/\|w_n\|_\infty.
\ee
Passing to the limit in \eqref{e4.9e} and using the boundedness of $f,g$, we immediately conclude that $v(t)$ is a nontrivial bounded full solution of \eqref{e4.5e} for $\lam=\lam_1$. Because for $\lam=\lam_1$, the equation \eqref{e4.5e} does not have any nontrivial bounded full solution other than the trivial one, one can obtain a contradiction.

Set
$$
  \cN=\mb{B}_{Y_2}(R_\ve)\X \cH,
  $$
where $\mb{B}_{Y_2}(R_\ve)$ denotes a ball in $Y_2$ centered at $0$ with radius $R_\ve$. Then by \eqref{e4.7e}, we know that $\cN$ is an admissible isolated neighborhood for the skew-product flow $\Phi_\lam^\nu$ generated by \eqref{e4.2e}. Thanks to the continuation property of Conley index (see (P4)), one obtains that
\be\label{e4.10e}
 h\big(\Phi_\lam,S_\8({\Phi_\lam})\big)= h\big(\Phi_\lam^{1},S_\8({\Phi^{1}_\lam})\big)=h\big(\Phi_\lam^{0},S_\8({\Phi^{0}_\lam})\big)
\ee
for $\lam\in[\ve_1,\mu-\ve]$, and
\be\label{e4.11e}
\ba{ll}
h\big(\Phi_\lam,S_\8({\Phi_\lam})\big)= h\big(\Phi_\lam^{1},S_\8({\Phi^{1}_\lam})\big)=h\big(\Phi_\lam^{0},S_\8({\Phi^{0}_\lam})\big)
\ea\ee
for $\lam\in[\mu+\ve,\ve_2]$. By virtue of Lemma \ref{l4.1}, we deduce that
\begin{align}\label{e4.12e}
  h(\Phi_\lam^{0},S_\infty(\Phi_\lam^{0}))=h(\vp_\lam,S_\infty(\vp_\lam))\wedge h(\theta,\cH)
\end{align}
for $\lam\in [\ve_1,\mu-\ve]$, and
\be\label{e4.13e}
  h(\Phi_\lam^{0},S_\infty(\Phi_\lam^{0}))=h(\vp_\lam,S_\infty(\vp_\lam))\wedge h(\theta,\cH)
\ee
for $\lam\in [\mu+\ve,\ve_2]$.
\vs
Now we distinguish $h(\vp_\lam,S_\infty(\vp_\lam))$ for $\lam\in [\ve_1,\mu-\ve]$ and $\lam\in [\mu+\ve,\ve_2]$, respectively. To the end, we study the equation
\be\label{e4.14e}
w_t+A^2 w=P_2\eta f(w+\xi_{\lam,0}(w)),\hs w\in Y_2,
\ee
where $\eta\in [0,1]$ is the homotopy parameter.
Repeating the same argument below \eqref{e4.7e}, one can easily show that for the fixed $\ve>0$ with $$\ve_1<\mu-\ve<\mu+\ve<\ve_2,$$ there exists $R'_\ve>0$ such that for every $\lam\in[\ve_1,\mu-\ve]\cup[\mu+\ve,\ve_2]$ and $\eta\in[0,1]$, if $w(t)$ is a bounded full solution of \eqref{e4.14e}, then
$$
\|w(t)\|_\infty<R'_\ve.
$$
Thus $\mb{B}_{Y_2}(R'_\ve)$ is an admissible isolated neighborhood for the semiflow $\vp_\lam^\eta$ generated by \eqref{e4.14e}. By the continuation of Conley index (see (P4)) and \eqref{e4.6e} we obtain that
\begin{align}\label{e4.15e}
  h(\vp_\lam,S_\infty(\vp_\lam))=h(\vp_\lam^1,S_\infty(\vp_\lam))
  =h(\vp_{\ve_1}^0,\{0\})=\Sig^0
\end{align}
for $\lam\in [\ve_1,\mu-\ve]$, and
\be\label{e4.16e}
   h(\vp_\lam,S_\infty(\vp_\lam))=h(\vp_\lam^1,S_\infty(\vp_\lam))
  =h(\vp_{\ve_2}^0,\{0\})=\Sig^q
\ee
for $\lam\in [\mu+\ve,\ve_2]$. Therefore, it follows from \eqref{e4.10e}-\eqref{e4.13e} and \eqref{e4.15e}-\eqref{e4.16e} that
\be\label{e4.17e}
 h\big(\Phi_\lam,S_\8({\Phi_\lam})\big)= \Sig^0\wedge h(\theta,\cH),\hs \lam\in [\ve_1,\mu-\ve],
\ee
and
\be\label{e4.18e}
   h\big(\Phi_\lam,S_\8({\Phi_\lam})\big)= \Sig^q\wedge h(\theta,\cH),\hs \lam\in [\mu+\ve,\ve_2].
\ee
Because $H^*(\cH)$ is free and finitely generated, one can easily see that $\~H^*(h(\theta,\cH))$ is free and finitely generated as well.
Applying $\~H^*$ to \eqref{e4.17e} and \eqref{e4.18e} and in view of the Lemma \ref{le4.3} in the Appendix, we conclude that the Conley indices in \eqref{e4.17e} and \eqref{e4.18e} are different.

By virtue of \cite[Theorem 4.5]{LW25}, we deduce that there exists a connected component $\Gam$ (consisting of bounded full solutions) of $\cS([\ve_1,\ve_2])$ for $\Phi_\lam$ meeting $Y_2\X \cH\X \{\ve_1,\ve_2\}$ such that $\Gam$ is unbounded in $Y_2\X \cH\X [\ve_1,\ve_2]$. Thus there are a sequence $\lam_n\in[\ve_1,\ve_2]$ and a sequence of bounded full solution $\~w_n\in \Gam[\lam_n]$ with
\be\label{e4.19e}
 \~w_n(t)=\Phi_{\lam_n}(t)(w_0,p_0)=(w_{\lam_n}(t,0;w_0,p_0),\theta_tp_0)
 \ee
for some $(w_0,p_0)\in Y_2\X \cH$, such that $\~w_n$ is unbounded in $Y_2\X \cH$ as $n\ra \8$.
\vs
On the other hand, we infer from \eqref{e4.7e} that for each fixed $\ve>0$,
$$\Gam[\lam_n]\ss \mb{B}_{Y_2}(R_\ve)\X \cH,\Hs \lam_n\in [\ve_1,\mu-\ve]\cup[\mu+\ve,\ve_2].$$
Therefore, we necessarily have $\lam_n\ra \mu$ as $n\ra \8$. Since $\cH$ is compact, one can conclude from \eqref{e4.19e} that $w_n(t)=w_{\lam_n}(t,0;w_0,p_0)$ is the desired sequence of bounded full solutions fulfilling the requirements in the theorem.
The proof of the theorem is complete.
\eo

\subsection{Bifurcation for abstract evolution equation}

\hs \,\, By virtue of Theorem \ref{t4.1}, we can give the corresponding bifurcation results of the original system \eqref{e3.7e} on the invariant manifold.

Define
$$
u_n(t)=w_n(t)+\xi_{\lam,p}(w_n(t)),
$$
where $w_n(t)$ is the sequence of bounded full solutions of the reduced equation \eqref{e4.1e} given by Theorem \ref{t4.1}.
According to the proof of the invariance for the manifold $\cM_{\lam,p}$ in Theorem \ref{t3.1}, one can easily know that $u_n(t)$ is a sequence of bounded full solutions of the original equation \eqref{e3.7e} contained in $\cM_{\lam,p}$. Moreover, we have the following results.

\bt\label{t4.2}
Assume the conditions {\rm({\bf F})} and {\rm({\bf H})} hold. Let $\mu\in \sig_2$. Then  $(\infty,\mu)$ is a bifurcation point of \eqref{e3.7e}. Specifically,
there exist a sequence $\lam_n$ with $\lam_n\ra \mu$ {\rm($n\ra \8$)} and a sequence of bounded full solutions $u_{n}=u_{\lam_n}(t,0;u_0,p)$ such that
$$
  \|u_{n}\|_\8\ra \8\hs \mb{as}\hs n\ra \8.
  $$
\et

\setcounter {equation}{0}
\section{Bifurcation of nonautonomous parabolic equation on unbounded domains}

\hs \,\, In this section we study the nonautonomous parabolic equation on unbounded domains:
\be\label{e5.1e}
u_t-\Delta u+v(x)u=\lam u+f(x,u)+g(x,t),\hs x\in \mathbb{R}^N,
\ee
where $v\in L^\infty(\R^N)$, $N\geq 1$, $\lam\in \R$ is the bifurcation parameter, and $f$ satisfies some Lipschitz condition and the  Landesman-Lazer type condition \eqref{LL}.

\subsection{Preliminaries}

\hs\,\, Denote $X=L^2(\mathbb{R}^N)$ and $Y=H^1(\mathbb{R}^N) (N\geq 1)$. Let $(\cdot,\cdot)$
and $|\cdot|$ denote the usual inner product and norm on $X$, respectively. The norm $\|\cdot\|$ on $Y$ is defined by
$$
  \|u\|=\big(\int_{\mathbb{R}^N}|\nabla u|^2\mathrm{d}x+\int_{\mathbb{R}^N}v(x)|u|^2\mathrm{d}x\big)^{1/2},\Hs
  u\in Y.
$$

In the following we further make some assumptions on the potential $v$, the nonlinearity $f$  and the function $g$ in \eqref{e5.1e}.
 \benu
\item[($\mathbf{A1}$)] There exist positive numbers $a_0$ and $v_\infty$ such that
$$
a_0\leq v(x),\hs 0<v_\infty:=\sup_{x\in\R^N}v(x)=\lim\limits_{|x|\ra \infty}v(x)<\infty.
$$
\item[($\mathbf{A2}$)] The nonlinear term $f:\R^N\X \R \ra \R$ satisfies
$$
 |f(x,s_1)-f(x,s_2)|\leq l(x)|s_1-s_2|,\Hs \forall s_1,s_2\in \R, \hs x\in \R^N,
$$
where $l(x)=l_1(x)+l_2(x)$, $l_1\in L^\8(\R^N)$, and $l_2$ satisfies
$$l_2\in L^p(\R^N),\hs\mb{$p\geq 2$ if $N=1$,\, $p>2$ if $N=2$, \,$p\geq N$ if $N\geq 3$.}$$
\item[($\mathbf{A3}$)] The function $f$ satisfies the Landesman-Lazer type condition \eqref{LL} in Section 1, and
$$
  |f(x,s)|\leq h(x), \Hs \A s\in\R,\hs x\in \R^N
$$
for some function $h\in L^2(\R^N)\cap L^\8(\R^N)$.
\item[({\bf G})] The function $g$ satisfies that
$$
  -\frac{1}{2}\ol{f}<\inf_{\R^N\X \R}g(x,t)\leqslant \max_{\R^N\X \R}g(x,t)<\frac{1}{2}\ul{f},
  $$
where $\ol{f}$ and $\ul{f}$ are the constants given in \eqref{LL}.
\eenu

By the condition ({\bf A1}), we deduce that the elliptic operator $A=-\Delta+v(x):H^2(\mathbb{R}^N)\ra L^2(\mathbb{R}^N)$ is a sectorial operator, and is selfadjoint and bounded from below. Moreover, one can see from \cite{BS,RS,WX} that the interval $[v_\infty,\infty)$ is the essential spectrum of $A$ and the discrete spectrum of $A$ on $(-\8,v_\8)$ appears. Specifically, for each fixed $a<v_\8$, $\sig(A)\cap (-\8,a)$ consists of at most finitely many eigenvalues of $A$. Thus, the spectrum $\sig(A)$ can be represented as
$$
  \sig(A)=\sig_p\cup\sig_e,
  $$
where $\sig_p$ consists of isolated eigenvalues with finite multiplicity, and $\sig_e=[v_\infty,\infty)$ is the essential spectrum.

Let $\mu\in \sig_p$ be fixed. Then the spectrum $\sigma(A)$ of $A$ can be decomposed as $\sig(A)=\sig_1\cup\sig_2\cup\sig_3$ with
$$
  \sig_1=\sig(A)\cap\{{\rm Re}\lam\leq \beta_1\},\hs \sig_2=\{\mu\},\hs \sig_3=\sig(A)\cap\{{\rm Re}\lam\geq \beta_2\},
$$
for some real numbers $\beta_1,\beta_2$. Clearly,
$
  {\rm Re}\,\sig(A)\cap(\beta_1,\beta_2)=\{{\rm Re}\,\mu\}.
$
Moreover, the space $X$ has a decomposition
$$
  X=X_1\oplus X_2 \oplus X_3,
$$
where both of the spaces $X_1$ and $X_2$ are finite-dimensional subspace of $X$. Let
$$
  P_i: X \ra X_i,\hs i\in\{1,2,3\}
  $$
be the projection from $X$ to $X_i$. Set
$$
  Y_i=Y\cap X_i,\Hs i\in\{1,2,3\}.
  $$
Then by the finite dimensionality of $X_1$ and $X_2$, one finds that $Y_1$ and $Y_2$ coincide with $X_1$ and $X_2$,
respectively. Also, it holds that
$$
  Y=Y_1\oplus Y_2\oplus Y_3.
  $$

Let $\beta$ be the positive number with \eqref{beta} and $J=({\rm Re}\mu-\frac{1}{4}\beta,{\rm Re}\mu+\frac{1}{4}\beta)$. Suppose that $g\in C_b(\R,X)$. Then the equation \eqref{e5.1e} can be written as an abstract cocycle system in $Y$:
\be\label{e5.3e}
u_t+Au=\lam u+ \~f(u)+p(t), \Hs p\in \cH,
\ee
where $\lam\in J$ and $\~f(u)$ denotes the Nemitski operator from $Y$ to $X$ given by
$$
  \~f(u)(x)=f(x,u),\Hs u\in Y.
$$
By ($\mathbf{A3}$), one knows that $\~f$ is well defined. Furthermore, we deduce from ($\mathbf{A2}$) that $\~f$ is global Lipschitz continuous from $Y$ to $X$. Recalling that $A$ is a sectorial operator, we conclude from \cite[Theorem 3.3.3]{Hen} that the Cauchy problem of \eqref{e5.3e} is well-posed. Namely, for each $u_0\in Y$, $t_0\in \R$ and $p\in \cH$, there exists $T>t_0$, such that \eqref{e5.3e} has a unique solution  $u(t)=u(t,t_0;u_0,p)$ on $[t_0,T)$  with $u(t_0)=u_0$. It is trivial to see from ($\mathbf{A3}$) that the solution $u(t)$ globally exists for $t\geqslant t_0$, $u_0\in Y$, $p\in \cH$.
Set
$$
 \psi(t,p)u_0=u(t,0;u_0,p),\Hs u_0\in Y,\hs p\in \cH,\hs t\geqslant 0.
 $$
Then $\psi$ is a cocycle semiflow on $Y$ driven by the translation group $\theta$ on the base space $\cH$.
\vs
In order to ensure the existence of nonautonomous global invariant manifolds for \eqref{e5.3e}, we
assume that the following condition on $f$ holds true.

\noindent($\mathbf{F1}$) The Lipschitz constant $L_f$ of $\~f$ satisfies
$$
   L_f M \int^\infty_0(2+s^{-\frac{1}{2}})\e^{-\frac{1}{4}\beta s}\d s<1.
  $$

Let
$$
 Y_{i,j}=Y_i\oplus Y_j,
$$
where $i,j=1,2,3,i\neq j$. Then according to Theorem \ref{t3.1}, we can obtain that the following conclusion.

\begin{thm}\label{t5.1}
Let the assumptions {\rm(}$\mathbf{A1}${\rm)}-{\rm(}$\mathbf{A3}${\rm)} and {\rm(}$\mathbf{F1}${\rm)} hold. Then for each $\lam\in J$ and each $p\in\cH$, there exists a Lipschitz continuous mapping $$\xi_{\lam,p}:Y_2\ra Y_{13},$$  such that the system \eqref{e5.3e} has a global invariant manifold $\cM_{\lam,p}$, defined by
$$
  \cM_{\lam,p}=\{w+\xi_{\lam,p}(w):w\in Y_2\},\Hs \lam\in J,\,\,p\in\cH.
$$
\end{thm}
Similar to \eqref{e4.1e}, one can restrict the original equation \eqref{e5.3e} on the nonautonomous global invariant manifold $\cM_{\lam,p}$:
\be\label{e5.4e}
w_t+A^2 w=P_2\~f(w+\xi_{\lam,p}(w))+P_2p(t),\hs w\in Y_2,
\ee
where $w=P_2u$, $A^2=P_2A_\lam=P_2(A-\lam I)$, $\lam\in J=(\mu-\frac{1}{4}\beta,\mu+\frac{1}{4}\beta),p\in\cH.$ Let $\phi_\lam$ denote the cocycle semiflow generated by \eqref{e5.4e} and $\Phi_\lam$ denote the corresponding skew-product flow on the product space $Y_2\X \cH$.
\vs
Furthermore, if the function $g$ in \eqref{e5.1e} satisfies 
the condition ({\bf H}) (in Section 4), then we immediately deduce by Theorem \ref{t4.1} that the following results for \eqref{e5.4e} hold.
\bt\label{t5.1}
 Let conditions {\rm({\bf A1})}-{\rm({\bf A3})}, {\rm({\bf F1})} and {\rm({\bf H})} hold. Suppose $\mu\in \sig_p$. Then  $(\infty,\mu)$ is a bifurcation point of \eqref{e5.4e}. Specifically,
there exist a sequence $\lam_n$ with $\lam_n\ra \mu$ {\rm($n\ra \8$)} and a sequence of bounded full solutions $w_{n}=w_{\lam_n}(t,0;w_0,p)$ of \eqref{e5.4e} such that
$$
  \|w_{n}\|_\8\ra \8\hs \mb{as}\hs n\ra \8.
  $$
\et

\subsection{More detailed dynamic bifurcations}


\hs \,\,In this subsection, we make use of the Landesman-Lazer type condition \eqref{LL} and ({\bf G}) to give a more detailed discussion on the dynamic bifurcation from infinity of the reduced system \eqref{e5.4e}. 
\vs
Let $w$ be a function on $\R^N$. Denote by $w_{\pm}$ the positive and negative parts of $w$, respectively. Namely,
$$
  w_{\pm}=\max\{\pm w(x),0\}, \Hs x\in \R^N.
  $$
Then $w=w_+-w_-.$
 \vs

\bl\label{l4.2}(\!\!\cite{LW21}) Let the assumption {\rm($\mathbf{A3}$)} hold true. Then
for each $R>0$ and $\varepsilon>0,$ there exists $s_0>0$ such that if $s\geq s_0$,
$$
  \int_{\mathbb{R}^N}f(x,v+sw)w{\rm d}x\geq
  \frac{1}{2}\int_{\mathbb{R}^N}(\bar{f}w_++\underline{f}w_-){\rm d}x-\varepsilon,
  $$
for all $w\in {\rm\bar{B}}_{Y^2}(1)$ and $v\in {\rm\bar{B}}_X(R)$, where ${\rm B}_X(R)$ denotes a ball in $X$ centered at $0$ with radius $R$.
\el

For each $0\leq a\leq b\leq \infty,$ put
$$
  \Xi[a,b]=\{z\in Y_2, a\leq |z|\leq b\}.
$$

\bl\label{l4.3}
Assume conditions {\rm ($\mathbf{A1}$)-($\mathbf{A3}$)}, {\rm({\bf F1})}, {\rm ({\bf H})} and {\rm({\bf G})} hold. Then there exist positive numbers $R_0,c_0$ such that the following results hold.
\benu
\item[(1)] For each $\lam\in[\mu,\mu+\frac{1}{4}\beta)$, if $w(t)$ is a solution of \eqref{e5.4e} in $\Xi[R_0,\8]$, then
\be\label{e4.20e}
\frac{{\rm d}}{{\rm d}t}|w(t)|^2\geq c_0 |w(t)|.
\ee
\item[(2)] For each $R>R_0$, there exists $\de\in(0,\frac{1}{4}\beta)$ such that if $\lam\in[\mu-\de,\mu)$, then \eqref{e4.20e} holds for every solution $w(t)$ of \eqref{e5.4e} in $\Xi[R_0,R]$.
 \item[(3)] There is an $\epsilon>0$ such that for each $\lam\in[\mu-\epsilon,\mu)$, the system $\Phi_\lam$ has a
 positively invariant absorbing set $\Xi[r_\lam,R_\lam]\X \cH$ satisfying
 $$
   r_\lam\ra\8\hs \mb{and}\hs R_\lam\ra \infty, \hs \mb{as} \hs\lam\ra \mu^-.
   $$
 \eenu
\el

\begin{proof}
Taking the inner product of \eqref{e5.4e} with $w\in Y_2$ in $X$, it yields that
\begin{align*}\label{e4.21e}
\frac{1}{2}\frac{{\rm d}}{{\rm d}t}|w|^2+\|w\|^2=\lam|w|^2+(\tilde{f}(w+\xi_{\lam,p}(w)),w)+(p(t),w).
\end{align*}
Since $\|w\|^2=\mu |w|^2$, it follows that
\be\label{e4.21e}
\frac{1}{2}\frac{{\rm d}}{{\rm d}t}|w|^2=(\lam-\mu)|w|^2+(\tilde{f}(w+\xi_{\lam,p}(w)),w)+(p(t),w).
\ee

In the following we estimate the last two terms in the above equation. Observing that the norm $\|\cdot\|_{L^1(\mathbb{R}^N)}$ of $L^1(\R^N)$ and that of $X=L^2(\mathbb{R}^N)$ are equivalent in the finite-dimensional space $Y_2,$ we deduce that
$$
\min\{\|z\|_{L^1(\mathbb{R}^N)}:\,\,z\in Y_2,\,\,|z|=1\}:=\kappa>0.
$$
In view of the Remark \ref{r3.1}, there exists $R_1>0$ such that $\|\xi_{\lam,p}(w)\|\leq R_1$, uniformly with respect to $p\in \cH$.
By the condition ({\bf G}), one can pick a positive number $\de$ with
$$
  \frac{1}{2}\bar{f}+p(x,t)\geq \de,\hs \frac{1}{2}\underline{f}-p(x,t)\geq \de.$$
From Lemma \ref{l4.2}, we infer that there exists $s_0>0$ such that if $s\geq s_0$,
\be\label{e4.22e}
(\~f(h+sz),z)=\int_{\mathbb{R}^N} f(x,h+sz)z\,{\rm d}x\geq
\frac{1}{2}\int_{\mathbb{R}^N}\(\ol fz_++\ul fz_-\){\rm d}x-\frac{1}{2}\kappa\de
\ee
for all $z\in \ol\mB_{Y_2}(1)$ and $h\in\ol\mB_X(R_1)$.
Set
$$w=sz,\hs\mb{where $s=|w|$}.$$
Then $z\in \partial \mathrm{B}_{Y_2}(1)$. Thus, if $s\geq s_0$, it follows from \eqref{e4.22e} that
\begin{align*}
&(\tilde{f}(w+\xi_{\lam,p}(w)),w)+(p(t),w)=s[(f(x,sz+\xi_{\lam,p}(w)),z)+(p(x,t),z)] \\[1ex]
\geq& s\[\(\frac{1}{2}\int_{\mathbb{R}^N}\(\ol f z_++\ul
f z_-\){\rm d}x-\frac{1}{2}\kappa\de \)+\int_{\mathbb{R}^N}\(p(x,t)z_+-p(x,t)z_-\)\d x\]\\[1ex]
=&s\[\int_{\mathbb{R}^N}\(\frac{1}{2}\ol{f}+p(x,t)\)z_++\(\frac{1}{2}\ul{f}-p(x,t)\)z_-\d x-\frac{1}{2}\kappa\de\].
\end{align*}
It is trivial to see that
\begin{align*}
&\int_{\mathbb{R}^N}\(\frac{1}{2}\ol f+p(x,t)\)z_++\(\frac{1}{2}\ul f-p(x,t)\)z_-\d x -\frac{1}{2}\kappa\de \\[1ex]
\geq&\de \int_{\mathbb{R}^N}|z|{\rm d}x-\frac{1}{2}\kappa\de\geq \frac{1}{2}\kappa\de.
\end{align*}
Hence,
\be\label{e4.23e}
(\tilde{f}(w+\xi_{\lam,p}(w)),w)+(p(t),w)\geq
\frac{1}{2}\kappa\de s=\frac{1}{2}\kappa\de |w|.
\ee
Combining \eqref{e4.21e} and \eqref{e4.23e}, we have
\be\label{e4.24e}
\frac{\d}{\d t}|w(t)|^2\geq 2\(\lam-\mu\) |w|^2+\kappa\de |w(t)|,
\ee
provided $|w(t)|\geq s_0$.
\vs
Let $R_0=s_0$,  $c_0=\kappa\de/2$. Then if $\lam\in[\mu,\mu+\frac{1}{4}\beta)$, we see from \eqref{e4.24e} that
$$
\frac{\d}{\d t}|w(t)|^2\geq \kappa\de |w(t)|>c_0|w(t)|
$$ at any point $t$ where $|w(t)|\geq R_0$, which shows the validity of assertion (1).

Assume $R>R_0$ and that $\lam<\mu$. Pick an $\eta>0$ such that $\eta R^2<\kappa\de s_0/4$. Then if $\lam\in[\mu-\eta,\mu)$ and $w(t)$ is a solution of \eqref{e5.4e} in $\Xi[R_0,R]$, we conclude from \eqref{e4.24e} that
\begin{align*}
\frac{\d}{\d t}|w(t)|^2&\geq -2|\lam-\mu|\,R^2+\kappa\de |w(t)|\\
&\geq c_0|w(t)|+\(c_0|w(t)|-2\eta R^2\)\\
&\geq c_0|w(t)|+ \(c_0s_0-2\eta R^2\) \geq c_0|w(t)|.
\end{align*}
Thus the assertion (2) holds.

Now, it remains to prove that the validity of assertion (3). Pick a sequence $R_n \,(n=1,2,\cdots)$ such that $R_0<R_1<R_2<\cdots<R_n<\cdots$ and $R_n\ra \8$ as $m\ra \8$. Then for every fixed $n$, by assertion (2) we know that there is a sequence $\eta_n>0 \,(n=1,2,\cdots)$ satisfying
$$\eta_1>\eta_{2}>\cdots>\eta_n>\cdots,\hs \eta_n\ra 0\,(n\ra \8),$$
such that for each $\lam\in [\mu-\eta_n,\mu)$, if $w(t)$ is a solution of \eqref{e5.4e} in $\Xi[R_0,R_n]$, then \eqref{e4.20e} holds. Moreover, we deduce from the choice of $\eta_n$ and \eqref{e4.20e} that for $\lam\in [\mu-\eta_n,\mu)$, if $w(t)=w(t,0;w_0,p)$ is a solution of \eqref{e5.4e} with $w_0\in \Xi[R_n,\8]$, then
$$
 w(t)\in \Xi[R_n,\8],\hs \A t\geq 0,
 $$
uniformly for $p\in\cH$.

On the other hand, let $\lam<\mu$. Recalling that the norm $\|\cdot\|_{L^1(\mathbb{R}^N)}$ of $L^1(\R^N)$ and that of $X=L^2(\mathbb{R}^N)$ are equivalent in $Y_2$, by the boundedness of $f$ and $g$, we find that
\begin{align}\label{e4.25e}
\big(\~f(w+\xi_{\lam,p}(w))+p(t),w\big)&=\int_{\R^N} \[f(x,w+\xi_{\lam,p}(w))+p(x,t)\]w\d x \no \\
&\leq \(f_\8+g_\8\)\|w\|_{L^1(\R^N)}\leq C_2|w|\no \\
&\leq \frac{\mu-\lam}{2}|w|^2+\frac{1}{2(\mu-\lam)}C_2^2,
\end{align}
where $f_\8$ and $g_\8$ denote the boundedness of $f$ and $g$, respectively, and $C_2>0$ depends on $f_\8,g_\8$.
Thus it follows from \eqref{e4.21e} and \eqref{e4.25e} that
$$
\frac{\d}{\d t}|w(t)|^2\leq (\lam-\mu)|w|^2+\frac{C_2^2}{\mu-\lam}.
$$
Applying the Gronwall's inequality, it yields that
\be\label{e4.26e}
  |w(t)|^2\leq \e^{-(\mu-\lam)t}|w(0)|^2+\big(1-\e^{-(\mu-\lam)t}\big)\frac{C_2^2}{(\mu-\lam)^2},\hs t\geq 0.
\ee
Let
$$
  R_\lam=\frac{C_2}{\mu-\lam}.
  $$
Then we see from \eqref{e4.26e} that if $w(0)\in \Xi[0,R_\lam]$, the corresponding solution
$$
  w(t)=w(t,0;w(0),p)\in \Xi[0,R_\lam], \hs t\geq 0,
$$
uniformly with respect to $p\in \cH$.

Since $\eta_n\ra 0$ as $n\ra \8$, one knows that
$$
  [\mu-\eta_1,\mu)=\Cup_{n\geq 1}[\mu-\eta_n,\mu-\eta_{n+1}).
$$
Take $\epsilon=\eta_1$. Clearly, if $\lam\in [\mu-\epsilon,\mu)$, then there is some $n$ such that $\lam\in [\mu-\eta_n,\mu-\eta_{n+1})$. By the choice of $\eta_n$, one deduces that
$$
  R_n\ra \8\Hs \mb{as \hs $\lam\ra\mu^-$}.
  $$
Therefore, let $r_\lam=R_n$. If $w_0\in\Xi[r_\lam,R_\lam]$, we deduce from the above argument that
\be\label{uni}
  w(t)=w(t,0;w_0,p)\in\Xi[r_\lam,R_\lam],
\ee
uniformly with respect to $p\in \cH$. The proof of this lemma is complete.
\end{proof}

By the above Lemma \ref{l4.3}, we obtain the following results on the Conley index of maximal compact invariant sets for \eqref{e5.4e}.

\bco\label{c4.1}
Let assumptions {\rm ($\mathbf{A1}$)-($\mathbf{A3}$)}, {\rm ({\bf F1})}, {\rm ({\bf H})} and {\rm({\bf G})} hold. Then $S_\8(\Phi_\lam)$ is uniformly bounded in $Y_2\X \cH$ with respect to $\lam\in [\mu,\mu+\frac{1}{4}\beta)$. Moreover,
\be\label{e4.27e}
h(\Phi_\mu,S_\8(\Phi_\mu))=h(\Phi_\lam,S_\8(\Phi_\lam))=\Sig^q\wedge h(\theta,\cH)
\ee
for $\lam\in [\mu,\mu+\frac{1}{4}\beta)$, where $q>0$ is given by \eqref{e4.6e}.\eco
\bo
Let $R_0$ be given in Lemma \ref{l4.3}. By virtue of  Lemma \ref{l4.3} (1), we find that
$$S_\8(\Phi_\lam)\ss \Xi[0,R_0]\X \cH,\hs \lam\in [\mu,\mu+\frac{1}{4}\beta).$$
Note that the constant $\ve$ in \eqref{e4.17e}-\eqref{e4.18e} can be arbitrary small. Since $\ve_1\in(\mu-\frac{1}{4}\beta,\mu)$ and $\ve_2\in (\mu,\mu+\frac{1}{4}\beta)$, we conclude from \eqref{e4.17e}-\eqref{e4.18e} that
\be\label{e4.28e}
 h\big(\Phi_\lam,S_\8({\Phi_\lam})\big)= \left\{\ba{ll}\Sig^0\wedge h(\theta,\cH),\hs \lam\in (\mu-\frac{\beta}{4},\mu), \\[1ex]
\Sig^q\wedge h(\theta,\cH),\hs \lam\in (\mu,\mu+\frac{\beta}{4}).\ea\right.\ee
Thus \eqref{e4.27e} follows by the continuation of Conley index.
\eo

Now we establish our main results on the dynamic bifurcation from infinity of the reduced equation \eqref{e5.4e} in this subsection.

\begin{thm}\label{t4.2}
Let the assumptions {\rm ($\mathbf{A1}$)-($\mathbf{A3}$)}, {\rm ({\bf F1})}, {\rm ({\bf H})} and {\rm({\bf G})} hold. Assume $\mu\in \sig_p$. Then $K_\lam=S_\8(\Phi_\lam)$ is nonempty for each $\lam\in J$. Furthermore, there exists $\eta>0$ such that the following results hold:
\benu
\item[(1)] For every $\lam\in \Lam_1:=[\mu-\eta,\mu)$, $K_\lam$ has a Morse decomposition $\sM=\{K_\lam^1,K_\lam^\8\}$.
\item[(2)] $K_\lam^1$ is bounded on $\Lam_1$, while
    \be\label{e4.29e}
    \lim_{\lam\ra\mu^-}\min_{\~w\in K_\lam^\8}\|w\|=\8,
    \ee
    where $\~w=(w(t,0;w_0,p),\theta_tp)$. Moreover, $K_\lam^\8$ is an attractor for the skew-product flow $\Phi_\lam$.
\item[(3)] Both of the sets $\sS^1$ and $\sS^\8$ have a component $\Gam$ satisfying $\Gam[\lam]\neq \emp$ for each $\lam\in\Lam_1$, where
$$\sS^1=\ol{\Cup_{\lam\in\Lam_1}(K_\lam^1\X\{\lam\})},\hs \sS^\8=\ol{\Cup_{\lam\in\Lam_1}(K_\lam^\8\X\{\lam\})}.$$
\eenu
\end{thm}
\begin{proof}
From Corollary \ref{c4.1} and \eqref{e4.7e}, we infer that
$K_\lam=S_\8(\Phi_\lam)$ is compact for $\lam\in J=(\mu-\frac{1}{4},\mu+\frac{1}{4}\beta)$. By Theorem \ref{th0.1}, \eqref{e4.27e} and \eqref{e4.28e}, we see that the Conley index of $K_\lam$ is nontrivial. Thus $K_\lam$ is nonempty.

(1)
Let $R_0$ be the number given by Lemma \ref{l4.3}. Since $K_\mu$ is isolated for $\Phi_\mu$, we can take a bounded isolated neighborhood $\cN_1$ of $K_\mu$ so that
\be\label{e4.30e}
K_\mu\ss \Xi [0,R_0]\X\cH\ss \cN_1.
\ee
Then there exists $\de_1\in (0,\frac{\beta}{4})$ such that $\cN_1$ is also an isolated neighborhood of $\Phi_\lam$ for $\lam\in\Lam:= [\mu-\de_1,\mu+\de_1]$. Thus
$$
 h(\Phi_\lam, K_\lam^1)=\mb{const.},\Hs \A \lam\in \Lam,
$$
where $K_\lam^1=S_\8(\Phi_\lam,\cN_1)$. By Corollary \ref{c4.1}, one has
\be\label{e4.31e}
h(\Phi_\lam, K_\lam^1)=\Sig^q\wedge h(\theta,\cH),\Hs \A \lam\in \Lam.
\ee
Let $c_0$ be the number given by Lemma \ref{l4.3}. Pick $R_1>R_0$ sufficiently large so that
\be\label{e4.32e}
\cN_1\ss\Xi[0,R_1]\X\cH.
\ee
By Lemma \ref{l4.3} (2), we know that there exists $\de>0$ with $\de<\de_1$ such that for each $\lam\in[\mu-\de,\mu)$, if $w(t)=w(t,0;w_0,p)$ is a solution of \eqref{e5.4e} in $\Xi[R_0,R_1]$, then
\be\label{e4.33e}
\frac{{\rm d}}{{\rm d}t}|w(t)|^2\geq c_0 |w(t)|\geq c_0R_0>0,
\ee
uniformly with respect to $p\in \cH$.

Pick $\eta\leq \de$. Then if $\lam\in\Lam_1:=[\mu-\eta,\mu)$ and $w(t)$ is a bounded full solution of \eqref{e5.4e} with $w(t_0)\in \Xi[R_0,R_1]$ for some $t_0$,
we can conclude from \eqref{e4.33e} that there is a $T_0>0$ such that
\be\label{e4.34e}
w(t)\in \Xi[0,R_0]\, (t<-T_0)\hs\mb{and}\hs w(t)\in \Xi[R_1,\8] \,(t>T_0),
\ee
uniformly for $p\in \cH$. Combining \eqref{e4.30e}, \eqref{e4.32e} and \eqref{e4.34e}, it yieds that
\be\label{e4.35e}
K_\lam^1\ss \Xi[0,R_0]\X\cH\ss \cN_1,\Hs\lam\in\Lam_1.
\ee
Because $\cN_1$ is an isolating neighborhood of $K_\lam^1$, by \eqref{e4.35e} one knows that $K_\lam^1$ is also the maximal compact invariant set of $\Phi_\lam$ in $\Xi[0,R_0]\X\cH$. This shows that $\Xi[0,R_0]\X\cH$ is an isolating neighborhood of $K_\lam^1$ as well.
Set
$$
  K_\lam^\8=S_\8(\Phi_\lam,\Xi[R_1,\8]\X\cH),\Hs\lam\in\Lam_1.
  $$
Clearly, $K_\lam^\8\ss K_\lam$.

Now we prove that $\sM=\{K_\lam^1,K_\lam^\8\}$ forms a Morse decomposition of $K_\lam$ for $\lam\in \Lam_1$. For this purpose, we first show that if $\~w(t)$ is a bounded full solution of $\Phi_\lam$ in $K_\lam\backslash(K_\lam^1\cup K_\lam^\8)$, then
\be\label{e4.36e}
\omega^*(\~w)\ss K_\lam^1, \Hs \omega(\~w)\ss K_\lam^\8.
\ee
Indeed, let $\~w(t)=(w(t),\theta_tp)$ be such a solution. It is easy to see that the maximal compact invariant set $K_\lam\ss \Xi[0,\8]\X\cH$ for $\lam\in\Lam_1$. Note that $K_\lam^1\cup K_\lam^\8\ss K_\lam$. Since $K_\lam^1$ and $K_\lam^\8$ are maximal compact invariant sets in $\Xi[0,R_0]\X\cH$ and $\Xi[R_1,\8]\X\cH$, respectively, one can deduce that there is $t_0$ such that $w(t_0)\in \Xi[R_0,R_1]$. Hence \eqref{e4.36e} follows by \eqref{e4.34e}.

 In the following we prove that $K_\lam^\8\neq \emp$, $\lam\in \Lam_1$, and hence $\sM$ is a Morse decomposition of $K_\lam$. Suppose the contrary. Then by \eqref{e4.36e} one knows that $K_{\lam_0}=K_{\lam_0}^1$ for some $\lam_0\in\Lam_1$. Thus, it follows from \eqref{e4.28e} that
\be\label{e4.37e}
h(\Phi_{\lam_0},K_{\lam_0}^1)=h(\Phi_{\lam_0},K_{\lam_0})=\Sigma^0\wedge h(\theta,\cH).
\ee
Since $H^*(\cH)$ is free and finitely generated and $q\geqslant1$, by Lemma \ref{le4.3} and similar to the analysis for \eqref{e4.17e} and \eqref{e4.18e}, we obtain a contradiction from \eqref{e4.37e} and \eqref{e4.31e}.

(2) It can be assumed that the $\eta\leq \epsilon$, where $\epsilon$ is given by Lemma \ref{l4.3} (3). Clearly, $K_\lam^1$ is bounded on $\Lam_1=[\mu-\eta,\mu)$. Let $\lam\in \Lam_1$. Thanks to Lemma \ref{l4.3} (3), the skew-product flow $\Phi_\lam$ has a positively invariant absorbing set $\Xi_\lam:=\Xi[r_\lam,R_\lam]\X \cH$. Since $Y_2$ is finite-dimensional and $\cH$ is compact, one can conclude from the attractor theory (see e.g., \cite{Chep,CLR}) that the $\Phi_\lam$ has a global attractor
$$
  \cA_\lam^\8=\Cup_{p\in\cH}\big(A_\lam^\8(p)\X\{p\}\big),\hs \lam\in \Lam_1
  $$
in $\Xi_\lam$, where
$$A_\lam^\8(p)=\Cap_{s\geq 0}\ol{\Cup_{t\geq s}\phi_\lam(t,\theta_{-t}p)\Xi_\lam(\theta_{-tp})}.$$

In what follows we further show that $K_\lam^\8=\cA_\lam^\8$ for $\lam\in \Lam_1$. Let $\lam\in \Lam_1$. Since $\Xi_\lam\ss \Xi[R_1,\8]\X\cH$, one can easily see that $\cA_\lam^\8\ss K_\lam^\8$.
Assume $\~w(t)=(w(t),\theta_tp)\in K_\lam^\8$ is a bounded full solution of $\Phi_\lam$. Then $\~w(t)\in \Xi[R_1,\8]\X\cH$. Recalling $w(t)=w(t,0;w_0,p)$ satisfies \eqref{e4.26e} uniformly with respect to $p\in\cH$, we conclude that $\~w(t)\in \Xi[R_1,R_\lam]\X\cH$.

Now we show that $\~w(t)$ is necessarily contained in $\Xi[r_\lam,R_\lam]\X\cH$. Suppose the contrary. Then there would be some $t_0\in \R$ such that $w(t_0)\in \Xi[R_1,r_\lam]$. Since $R_1>R_0$, one finds that $w(t_0)\in \Xi[R_0,r_\lam]$. Using the same argument of \eqref{e4.34e}, it can be easily shown that there exists $T_2>0$ such that
$$w(t)=w(t,0;w_0,p)\in \Xi[0,R_0],\hs t<-T_2.$$
This leads to a contradiction. Thus $\~w(t)\in \Xi[r_\lam,R_\lam]\X\cH$, which implies that $K_\lam^\8\ss \cA_\lam^\8$.

To verify the validity of assertion (2), it remains to prove that \eqref{e4.29e} holds. Note that $K_\lam^\8\ss \Xi[r_\lam,R_\lam]\X\cH$. Let
$\~w(t)=(w(t),\theta_tp)\in K_\lam^\8$ be a full solution. Then we deduce from \eqref{uni} that $w(t)=w(t,0;w_0,p)\in \Xi[r_\lam,R_\lam]$ uniformly with respect to $p\in \cH$. Since
$$
 r_\lam\ra \8\hs \mb{and}\hs R_\lam\ra \8\Hs \mb{as} \hs \lam\ra\mu^-,
 $$
we conclude that \eqref{e4.29e} holds true.

(3) Finally, we prove the validity of assertion 3. Let $\cU=\Xi[R_1,\8]\X \cH$. Then 
$$
  \pa \cU=\{w\in Y_2:|w|=R_1\}\X\cH.$$ 
For each $R>R_1$, from Lemma \ref{l4.3} (2) we infer that there exists $\de_2>0$ such that if $\lam\in [\mu-\de_2,\mu)$ and $w(t)$ is a solution of \eqref{e5.4e} in $\Xi[R_0,R]$, then \eqref{e4.20e} holds. It may be assumed that the $\eta<\de_2$. Thus if $\lam\in [\mu-\eta,\mu)$, then every bounded full solution in $\Xi[R_1,\8]\X\cH$ is necessarily contained in $\Xi[R,\8]\X\cH$,  which shows that
$$
 K_\lam^\8\ss\Xi[R,\8]\X\cH.
 $$
Hence $\cU$ is an isolating neighborhood of $K_\lam^\8$. It is trivial to see that
$$h(\Phi_\lam,K_\lam)=h(\Phi_{\lam},K^1_\lam)\vee h(\Phi_\lam,K^\8_\lam),\hs \lam\in \Lam_-=[\mu-\eta,\mu).$$
Then we infer from \eqref{e4.28e} and \eqref{e4.31e} that
$$h(\Phi_{\mu-\eta},K_\lam^\8)=h(\Phi_{\mu-\eta},\cU)\neq \ol{0}.$$
 Therefore, by \cite[Theorem 4.4]{LW25} (or \cite[Theorem 3.4]{LLZ}), one can conclude that $\sS^\8$ has a component $\Gam$ such that one of the assertions 1-3 in \cite[Theorem 4.4]{LW25} holds true. Since $\cU$ is an isolating neighborhood and $\Gam[\lam]\ss {\rm int}\,\cU$ for all $\lam\in \Lam_1$, we deduce that either $\Gam$ is unbounded or $\Gam[\mu]\neq \emp$ for $\lam\in \Lam_1$. Recalling that
 $$
  K_\lam^\8\ss \Xi [R_1,\8]\X\cH\hs \mb{and} \hs K_\mu\ss \Xi [0,R_0]\X\cH,
  $$
 one can immediately conclude that $\Gam$ is unbounded for $\lam\in \Lam_1$. Because
 $\Gam[\lam]$ is uniformly bounded on $[\mu-\eta,\mu-\ve]$ for every $\ve\in (0,\eta)$, we have $\Gam[\lam]\neq \emp$ for $\lam\in \Lam_1$.

The argument for $\sS^1$ is similar, we omit the details.
\end{proof}

By the proof of the above theorem, we can further show that if $\lam\in \Lam_1:=[\mu-\eta,\mu)$, then the set $K_\lam$ (in Theorem \ref{t4.2}) has a Morse decomposition 
and the past Morse decomposition in the sense of the works in Aragao-Costa, Caraballo, Carvalho and Langa \cite{ACCL} and Rasmussen \cite{Ras,Ras07}, respectively.
\bt\label{t4.3}
Let the conditions in Theorem \ref{t4.2} hold.  Then for each $\lam\in \Lam_1:=[\mu-\eta,\mu)$, the maximal compact invariant set $K_\lam$ of \eqref{e5.4e} has a Morse decomposition $M=\{\cA_\lam^\8,\cA_\lam^1\}$ in the sense of the works in \cite{ACCL} and \cite{Ras,Ras07}, respectively, where $\cA_\lam^\8$ satisfies \eqref{e4.29e} and $\cA_\lam^1$ remains bounded on $\Lam_1$.
\et
\bo
Let $\lam\in \Lam_1$. Note that $K_\lam^1=S_\8(\Phi_\lam,\cN_1)$ is the maximal compact invariant set of the skew-product flow $\Phi_\lam=(\phi_\lam,\theta)$ generated by \eqref{e5.4e} with $\cN_1$ being its isolating neighborhood. Let $\phi_\lam^{-1}$ denote the inverse flow of $\phi_\lam$. Then by the attractor theory (see e.g., \cite{CLR}), we deduce that $(\phi_\lam^{-1},\theta)$ has an attractor $\cA_\lam^1$ in $\cN_1$ satisfying
$$\cA_\lam^1=\Cup_{p\in\cH}\big(A_\lam^1(p)\X\{p\}\big).$$
Hence the family of sets $\{A_\lam^1(p)\}_{p\in \cH}$ is the pullback attractor for the skew-product flow $(\phi_\lam^{-1},\theta)$. Recalling that $\Xi[R_1,\8]\X\cH$ is an isolating neighborhood of $K_\lam^\8=\cA_\lam^\8$, we conclude from \eqref{e4.35e} that $M=\{\cA_\lam^\8,\cA_\lam^1\}$ is an isolated invariant families of $K_\lam$.  It is clear to see that
$$
  K_\lam\ss \Xi[0,\8]\X\cH\hs \mb{and} \hs\cA_\lam^\8\cup \cA_\lam^1\ss K_\lam.
  $$
Since $\cA_\lam^1$ and $\cA_\lam^\8$ are maximal compact invariant sets in $\Xi[0,R_0]\X\cH$ and $\Xi[R_1,\8]\X\cH$, respectively, we deduce that if
$$
  \~w(t)=(w(t),\theta_tp)=(w(t,0;w_0,p),\theta_tp)\in K_\lam\backslash(\cA_\lam^1\cup \cA_\lam^\8)
  $$
is a full solution of $\Phi_\lam$, then there exists $t_0$ such that $w(t_0)\in \Xi[R_0,R_1]$, and hence \eqref{e4.33e} and \eqref{e4.34e} hold. This shows that
$$
  \lim_{t\ra \8}\d_H(\phi_\lam(t)w_0,A_\lam^\8(\theta_tp))=0\hs\mb{and}\hs \lim_{t\ra -\8}\d_H(\phi_\lam(t)w_0,A_\lam^1(\theta_tp))=0,
  $$
where $\d_H(A,B)$ denotes the Hausdorff semidistance between $A$ and $B$. Therefore, one can conclude from the theories of the nonautonomous Morse decomposition in \cite{ACCL} and \cite{Ras,Ras07} that $M=\{\cA_\lam^\8,\cA_\lam^1\}$ forms a Morse decomposition and a past Morse decomposition of $K_\lam$, respectively for $\lam\in \Lam_1$. 
It is clear to see that $\cA_\lam^1$ remains bounded on $\Lam_1$. Since $K_\lam^\8=\cA_\lam^\8$ for $\lam\in \Lam_1$, it follows that $\cA_\lam^\8$ satisfies \eqref{e4.29e} as well. The proof is complete.
\eo
\br
By the proof of Theorems \ref{t4.2} and \ref{t4.3} we know that the Morse decomposition $M=\{\cA_\lam^\8,\cA_\lam^1\}$, $\lam\in\Lam_1$ satisfies $$\cA_\lam^1=\Cup_{p\in\cH}\big(A_\lam^1(p)\X\{p\}\big),\hs \cA_\lam^\8=\Cup_{p\in\cH}\big(A_\lam^\8(p)\X\{p\}\big),$$
where the family of sets $\{A_\lam^\8(p)\}_{p\in \cH}$ is a pullback attractor for the skew-product flow $(\phi_\lam,\theta)$.
\er

\setcounter {equation}{0}
\section{Bifurcation of the original equation}

\hs \, By virtue of Theorems \ref{t4.2} and \ref{t4.3}, we obtain the corresponding bifurcation results of the original system \eqref{e5.3e} on the invariant manifold $\cM_{\lam,p}$.

Let $w(t)=w(t,0;w_0,p)$ be the bounded full solution of \eqref{e5.4e}. Define
$$
\cK_\lam=\{w+\xi_{\lam,p}(w):\~w\in K_\lam\},\hs \sM=\{\cK_\lam^1,\cK_\lam^\8\},\hs \lam\in \Lam_1=[\mu-\eta,\mu),
$$
where $\~w=(w(t),\theta_tp)$, and
$$
  \cK_\lam^1=\{w+\xi_{\lam,p}(w):\~w\in K_\lam^1\},\hs\cK_\lam^\8=\{w+\xi_{\lam,p}(w):\~w\in K_\lam^\8\}.
  $$
Then using the same argument as the proof of invariance for $\cM_{\lam,p}$ in Theorem \ref{t3.1}, one can easily show that $\cK_\lam,\cK_\lam^1,\cK_\lam^\8$ are compact invariant sets of the original equation \eqref{e5.3e} contained in $\cM_{\lam,p}$. Furthermore, we have the following conclusions.
\begin{thm}\label{t6.1}
Let the assumptions {\rm ($\mathbf{A1}$)-($\mathbf{A3}$)}, {\rm ({\bf F1})}, {\rm ({\bf H})} and {\rm ({\bf G})} hold. Assume $\mu\in \sig_p$. Then $\cK_\lam$ is nonempty for every $\lam\in J$. Furthermore, there is an $\eta>0$ such that the following results holds:
\benu
\item[(1)] For every $\lam\in \Lam_1=[\mu-\eta,\mu)$, $\cK_\lam$ has a Morse decomposition $\sM_1=\{\cK_\lam^1,\cK_\lam^\8\}$.
\item[(2)] $\cK_\lam^1$ remains bounded on $\Lam_1$, while
    \be\label{e5.29e}
    \lim_{\lam\ra\mu}\min_{\~u\in \cK_\lam^\8}\|u\|=\8,
    \ee
    where $\~u=(u(t,0;u_0,p),\theta_tp)$. Moreover, $\cK_\lam^\8$ is an attractor for the skew-product flow $(\psi,\theta)$ generated by \eqref{e5.3e}.
\item[(3)] Both of the sets $\sC^1$ and $\sC^\8$ have a component $\Upsilon$ satisfying $\Upsilon[\lam]\neq \emp$ for each $\lam\in\Lam_1$, where
$$\sC^1=\ol{\Cup_{\lam\in\Lam_1}(\cK_\lam^1\X\{\lam\})},\hs \sC^\8=\ol{\Cup_{\lam\in\Lam_1}(\cK_\lam^\8\X\{\lam\})}.$$
\eenu
\end{thm}
\br\label{r6.1}
In \cite{CK}, \'Cwiszewski and Kryszewski studied the bifurcation from infinity of elliptic equation on $\R^N$ by giving some results of the corresponding parabolic equation \eqref{e5.1e} with $p=0$. Under the Landesman-Lazer type conditions (see \cite[p.3]{CK}) or the assumption that the isolated eigenvalue of $A$ is of odd multiplicity, they established a theorem on bifurcations from infinity of the elliptic equation. Theorem \ref{t6.1} significantly extends and improves the results in \cite{CK} in some sense.
\er
Finally, set
$$
  \sA_\lam^1=\{w+\xi_\lam(w):\~w\in \cA_\lam^1\},\hs\sA_\lam^\8=\{w+\xi_\lam(w):\~w\in \cA_\lam^\8\}.
  $$
Then by Theorem \ref{t4.3}, we also obtain the following results.

\bt\label{t5.2}
Let the conditions in Theorem \ref{t4.2} hold.  Then for each $\lam\in \Lam_1:=[\mu-\eta,\mu)$, the maximal compact invariant set $\cK_\lam$ of \eqref{e5.3e} has a Morse decomposition $M_1=\{\sA_\lam^\8,\sA_\lam^1\}$ in the sense of the work \cite{ACCL}, where the family of sets $\sA_\lam^\8$ is a pullback attractor for the original equation \eqref{e5.3e}.
\et

\section{Summary and remark}

\hs \, In this work, we study the dynamic bifurcation from infinity of the nonautonomous evolution equation \eqref{e3.1e} on unbounded domains by using the invariant manifold and the reduced singular cohomology groups method, which is different from those in the literature.

Firstly, we establish a nonautonomous global invariant manifold for the equation \eqref{e3.1e}, and then we reduce \eqref{e3.1e} to this invariant manifold, which generates a finite-dimensional nonautonomous system. Secondly, we construct a homotopy between the skew-product flow generated by the finite-dimensional reduced system and a product flow. This homotopy plays an important role in discussing the Conley index of invariant sets. Using this homotopy and some topological consequences on the reduced singular
cohomology groups (see the Appendix in \cite{LW25}), we establish some results on bifurcations from infinity of the reduced equation and the original equation. As an example, the nonautonomous parabolic equation \eqref{e1.3} is studied. Under the Landesman-Lazer type condition \eqref{LL}, we give more detailed descriptions (see Lemma \ref{l4.3}) on the dynamical behaviors of the reduced equation, and establish some precise results (see Theorems \ref{t4.2}, \ref{t4.3}) on the dynamic bifurcation from infinity of the reduced equation. Finally, the corresponding results on dynamic bifurcations from infinity of the original equation \eqref{e3.1e} are derived.
\vs
We remark that the approach, using the invariant manifold and the reduced singular cohomology groups to investigate the dynamic bifurcation from infinity for \eqref{e3.1e} can be applied to a vast body of nonautonomous and autonomous evolution equations on bounded or unbounded domains. Indeed, if the corresponding operator of a nonlinear evolution equation is a sectorial operator having an isolated eigenvalue of finite multiplicity, then our main results in this work still hold true under some suitable conditions on the nonautonomous term and the Landesman-Lazer type condition on nonlinear terms.

\section{Appendix: Topological consequences on Conley index}\label{s6}

\hs\, In the Appendix, for readers' convenience we collect some topological consequences on the reduced singular
 cohomology groups from \cite{Hat,WLD,LW25}.

 Assume that $(X,x_0)$ and $(Y,y_0)$ are two pointed spaces. By $H^*$ ($\~H^*$) we denote the (reduced) singular cohomology theories with coefficient group $\mathbb Z$.

It is well-known that the cross product $``\X"$ (\cite{Hat}) defined on the tensor product in the following
$$H^*(X,x_0)\otimes H^*(Y,y_0)\stac{\X}{\ra} H^*(X\X Y,\{x_0\}\X Y\cup X\X\{y_0\}).$$

$(X,x_0)$ is said to have the {\em homotopy extension property} (HEP for short), if $\{x_0\}$ is a strong deformation retract of one of its open neighborhoods in $X$.
First, we give the following conclusion.
\bt\label{th0.1}(\!\!\cite{LW25}) Let $(X,x_0)$ and $(Y,y_0)$ have the HEP. Assume that $H^k(Y,y_0)$ is free and finitely generated for each $k$.
Then the cross product homomorphism
$$H^*(X,x_0)\otimes H^*(Y,y_0)\stac{\X}{\ra} H^*(X\X Y,\{x_0\}\X Y\cup X\X\{y_0\})$$
is an isomorphism.\et


If $(X,x_0)$ and $(Y,y_0)$ have the HEP, we can easily see that $(X\X Y,\{x_0\}\X Y\cup X\X\{y_0\})$ has the HEP as well.
Similar to the case of CW complex in \cite{Hat}, then one deduces that the cross product for $(X,x_0)$ and $(Y,y_0)$ can give a {\em reduced cross product}
$$\~H^*(X)\otimes\~H^*(Y)\stac{\X}{\ra}\~H^*(X\wedge Y),$$
where $X\wedge Y:=(X,x_0)\wedge (Y,y_0)$.

Concerning this reduced cross product, by Theorem \ref{th0.1}, we deduce that if $\~H^*(X)$ or $\~H^*(Y)$ is free and finitely generated in each dimension, then the reduced cross product is an isomorphism.
\bl\label{l0.3}(\!\!\cite{LW25})
Suppose that the pointed spaces $(X_i,x_{i})$ and $(Y,y_0)$ have the HEP for $i=1,\,2$.
Assume that $\~H^*(X_1)\not\approx\~H^*(X_2)$ and that $\~H^*(Y)\neq0$.
Then if $\~H^*(X_i)$ and $\~H^*(Y)$ are free and finitely generated in each dimension, $i=1,\,2$, we have
\be\label{0.1}
\~H^*(X_1\wedge Y)\not\approx\~H^*(X_2\wedge Y).
\ee
\el
Assume $K$ is a compact isolated invariant set for a semiflow $\Phi$ and that $h(\Phi,K)$ denotes its Conley index.
Then applying $\~H^*$ to $h(\Phi,K)$, we can obtain the {\em reduced cohomology Conley index} of $K$.

Let $K_i$ be a compact isolated invariant set of the semiflow $\Phi_i$ on $X_i$ for $i=1,\,2$.
Then by the definition, one can obtain two Conley index pairs $(N_1,E_1)$ and $(N_2,E_2)$ of $K_1$ and $K_2$, respectively so that
$$\~H^*(h(\Phi_1,K_1))=\~H^*(N_1/E_1,[E_1]),\Hs \~H^*(h(\Phi_2,K_2))=\~H^*(N_2/E_2,[E_2]).$$
According to the property of index pairs and quotient flows (see \cite {WLD,WLD1}), we immediately conclude that $[E_i]$ is an attractor in $N_i/E_i$ for the quotient flow of $\Phi_i$, $i=1,\,2$.
Therefore $(N_i/E_i,[E_i])$ has the HEP and one can define the reduced cross product of reduced cohomology Conley indices in the following:
$$\~H^*(h(\Phi_1,K_1))\otimes\~H^*(h(\Phi_2,K_2)\stac{\X}{\ra}\~H^*(h(\Phi_1\X\Phi_2,K_1\X K_2)).$$

By virtue of Lemma \ref{l0.3}, one can obtain the following main result, which plays an important role in proving our conclusions.
\bl\label{le4.3}(\!\!\cite{LW25})
Let $K_i$ be a compact isolated invariant set of the semiflow $\Phi_i$ on $X_i$ for $i=1,\,2,\,3$.
Assume that $\~H^*(h(\Phi_i,K_i))$ is free and finitely generated for $i=1,\,2,\,3$, $\~H^*(h(\Phi_1,K_1))\not\approx\~H^*(h(\Phi_2,K_2))$ and that $\~H^*(h(\Phi_3,K_3))\neq0$.
Then we have
$$\~H^*(h(\Phi_1\X\Phi_3,K_1\X K_3))\not\approx\~H^*(h(\Phi_2\X\Phi_3,K_2\X K_3)).$$
\el
\Vs
\noindent{\bf Acknowledgments}
\Vs
This work is supported by the National Natural Science Foundation of China (Grant Nos. 12101462, 11801190).


\Vs
\noindent{\bf Declaration of competing interests}
\Vs
We have nothing to declare.

\medskip
\medskip

\end{document}